\documentclass[11pt]{amsart}
\usepackage{amsmath, amssymb, amscd, amsthm, amsfonts}
\usepackage{mathtools}

\usepackage{enumitem}
\usepackage[hidelinks]{hyperref}

\setlist[enumerate]{itemsep=2pt,parsep=2pt,before={\parskip=2pt}}

\usepackage[backend=bibtex,giveninits=true,maxbibnames=99, doi=false,isbn=false,url=false,eprint=false]{biblatex}

\bibliography{refs}

\renewbibmacro{in:}{}

\numberwithin{equation}{section}

\newtheorem{Thm}{Theorem}[section]
\newtheorem{Prop}[Thm]{Proposition}
\newtheorem{Lem}[Thm]{Lemma}
\newtheorem{Cor}[Thm]{Corollary}
\theoremstyle{definition}
\newtheorem{Expl}[Thm]{Remark}
\newtheorem{Ex}[Thm]{Example}
\newtheorem{Def}[Thm]{Definition}

\renewcommand{\phi}{\varphi}
\renewcommand{\rho}{\varrho}

\newcommand{\Q}{\mathbb{Q}}
\newcommand{\N}{\mathbb{N}}
\newcommand{\CAT}{\operatorname{CAT}}
\newcommand{\diam}{\operatorname{diam}}
\newcommand{\norm}[1]{\|#1\|}
\newcommand{\abs}[1]{\lvert#1\rvert}

\DeclareMathOperator{\Con}{Con}

\DeclareMathOperator{\R}{\mathbb{R}}

\DeclareMathOperator{\Lip}{Lip}

\DeclareMathOperator\Iso{Isom}
\DeclareMathOperator\CBI{CB}
\DeclareMathOperator\RCBI{RCB}

\DeclareMathOperator{\conv}{conv}
\DeclareMathOperator{\spt}{spt}

\title{Extending and improving conical bicombings}

\author{Giuliano Basso}
\address{Department of Mathematics, University of Fribourg, Chemin du Mus\'ee 23, 1700 Fribourg, Switzerland
\&
Max Planck Institute for Mathematics,
Vivatsgasse 7,
53111 Bonn,
Germany}
\email{basso@mpim-bonn.mpg.de}

\keywords{injective metric space, conical bicombing, Z-structure}

\subjclass[2020]{Primary 53C23; Secondary 20F65 and 51F99}

\begin{document}

\pdfbookmark[0]{Extending and improving conical bicombings}{titleLabel}

\begin{abstract}
We study metric spaces that admit a conical bicombing and thus obey a weak form of non-positive curvature. Prime examples of such spaces are injective metric spaces. In this article we give a complete characterization of complete metric spaces admitting a conical bicombing by showing that every such space is isometric to a \(\sigma\)-convex subset of some injective metric space.
In addition, we show that every proper metric space that admits a conical bicombing also admits a consistent bicombing that satisfies certain convexity conditions. This can be seen as a strong indication that a question from Descombes and Lang about improving conical bicombings might have a positive answer. As an application, we prove that any group acting geometrically on a proper metric space with a conical bicombing admits a \(\mathcal{Z}\)-structure.
\end{abstract}

\maketitle

%%%%%%%%%%%%%%%%%%%%%%%%%%%%%%%%%%%%%%%%%%%%%%%%%%%%%%%%%%%%%%%%%%%
\section{Introduction}
%%%%%%%%%%%%%%%%%%%%%%%%%%%%%%%%%%%%%%%%%%%%%%%%%%%%%%%%%%%%%%%%%%%
\subsection{Extending conical bicombings}

A metric space \(X\) is called \textit{injective} if it is an injective object in the category of metric spaces with \(1\)-Lipschitz maps as morphisms. More concretely, \(X\) is said to be injective if for any metric space \(B\), every 1-Lipschitz map \(f\colon A\to X\), \(A\subset B\), can be extended to a 1-Lipschitz map \(\bar{f}\colon B\to X\).  Injective metric spaces, also called hyperconvex metric spaces by some authors, were first studied by Aronszajn and Panitchpakdi in \cite{MR84762} and have since been applied in fields as diverse as functional analysis, geometric group theory, metric fixed point theory and phylogentic analysis. Particular examples of injective spaces are the real line, complete metric \(\R\)-trees and finite \(\CAT(0)\) cube complexes endowed with the length metric which is induced by choosing the \(\ell_\infty\)-norm on each cube (see \cite{MR695270} and also \cite{MR4043820, MR1452832} for related results).  Further examples are the Banach space \(\ell_\infty\) of bounded sequences equipped with the supremum norm, closed geodesically convex subsets of \(\ell_\infty\), and, as shown in \cite{MR3682664}, certain subsets of \(\ell_\infty\) that lie between graphs of \(1\)-Lipschitz functions. In contrast to these examples, however, a smooth Riemannian manifold is injective if and only if it is isometric to the real line. 

As observed by Lang in \cite{MR3096307}, injective metric spaces have striking properties reminiscent of non-positive curvature. In particular, on every injective metric space \(X\) there exist certain distinguished geodesics which satisfy a weak global non-positive curvature condition. More precisely, there exists a map \(\sigma\colon X\times X\times [0,1]\to X\) subject to the following conditions. The curve \(\sigma_{xy}\coloneqq \sigma(x,y, \cdot)\) is a constant speed geodesic from \(x\) to \(y\) and
\begin{equation}\label{eq:conicalDEF}
d(\sigma_{xy}(t), \sigma_{x'y'}(t)) \leq (1-t)\,d(x,x') +t\, d(y,y')
\end{equation}
for all \(x,y,x',y'\in X\) and all \(t\in [0,1]\). Following Lang, we call such a map \(\sigma\) a \textit{conical bicombing}. Recently, conical bicombings have become a useful tool in geometric group theory in connection with Helly groups (see \cite{cavallucci2021discrete}, \cite{chalopin2020helly}, \cite{haettel2021lattices}, \cite{haettel2020coarse}, \cite{MR4285138}) and in metric fixed point theory where various fixed point results which hold for convex subsets in Banach spaces have been transferred to spaces admitting a conical bicombing (see \cite{itoh1979some}, \cite{MR2578604}, \cite{MR1072312}). In the present article we continue with the study of conical bicombings which was initiated in \cite{MR3940917}, \cite{MR3370039}, \cite{MR3483604}, \cite{MR3777137}.

The class  \(\Con\) of all metric spaces admitting a conical bicombing enjoys many desirable structural properties. For example, it is closed under ultralimits, \(\ell_p\)-products, for \(p\in [1, \infty]\), and \(1\)-Lipschitz retractions. Let \(X\) be a member of \(\Con\). We say that \(A\subset X\) is \(\sigma\)-convex if there exists a conical bicombing \(\sigma\) on \(X\) such that for all \(x\), \(y\in A\) the geodesic \(\sigma_{xy}(\cdot)\) is contained in \(A\). Clearly, every \(\sigma\)-convex subset of \(X\) also belongs to \(\Con\). As alluded to above, injective spaces admit conical bicombings, and thus closed \(\sigma\)-convex subsets of injective spaces are examples of metric spaces admitting a conical bicombing. Our first result shows that these examples completely characterize the class of complete metric spaces that admit a conical bicombing.

\begin{Thm}\label{thm:main-1-neu}
Let \(X\) be a complete metric space. Then the following statements are equivalent:
\begin{enumerate}
\item \(X\) admits a conical bicombing.
\item \(X\) is isometric to a \(\sigma\)-convex subset of some injective metric space. 
\end{enumerate}
\end{Thm}

The main tools used to prove Theorem~\ref{thm:main-1-neu} are the \(1\)-Wasserstein distances from optimal transport theory and metric injective hulls (also known as tight-spans). All relevant material concerning \(1\)-Wasserstein distances can be found in Section~\ref{sec:section-two}. We continue with a short discussion of injective hulls. We follow \cite[Definition 9.12]{MR2240597} and call an isometric embedding \(i\colon X\to Y\) \textit{essential} provided that a \(1\)-Lipschitz map \(j\colon Y\to Z\) to any metric space \(Z\) is an isometric embedding, whenever \(j\circ i\) is an isometric embedding. A remarkable result of Isbell \cite{MR184061} states that every metric space \(X\) has an essentially unique injective hull \((E(X), i)\). By definition,  \(E(X)\) is an injective metric space and \(i\colon X\to E(X)\) an essential isometric embedding. For other equivalent descriptions of the injective hull we refer to \cite[Proposition 9.20]{MR2240597}. The existence of injective hulls has been rediscovered several times (see \cite{MR1258238}, \cite{MR753872}, \cite{MR0196709}). 

We refer to \cite[pp. 334 -- 339]{MR753872} for some pictures of the injective hulls of \(n\)-point metric spaces for small \(n\). It turns out that the injective hull of a finite metric space is always isometric to a finite polyhedral complex whose cells are subsets of \(\ell_\infty^d=(\R^d, \norm{\cdot}_\infty)\), where \(d\) is the greatest integer such that \(2d \leq \# X\). Moreover, injective hulls of \(0\)-hyperbolic spaces are metric \(\R\)-trees (see \cite[Theorem 8]{MR753872}) and in \cite{MR3096307} it is shown that the injective hulls of many interesting locally finite graphs are locally finite polyhedral complexes which have only finitely many isometry types of cells. We also note that injective hulls are increasingly used as a tool in geometric group theory. For example, they are used by Chalopin, Chepoi, Genevois, Hirai and Osajda, to show that every Helly group admits a geometric group action on an injective metric space (see \cite[Theorem 1.5]{chalopin2020helly}).

By construction, the injective hull \(E(X)\) is the `smallest' injective metric space containing \(X\). Indeed, if \(Y\) is an injective metric space and \(j\colon X \to Y\) an isometric embedding, then as \(i\colon X\to E(X)\) is essential, there exists an isometric embedding \(k\colon E(X) \to Y\) such that \(j=k\circ i\). In what follows, we will often tacitly identify \(X\) with its isometric copy \(i(X) \subset E(X)\). Due to the following extension result, the injective space appearing in Theorem~\ref{thm:main-1-neu} can be taken to be the injective hull of \(X\). 

\begin{Thm}\label{thm:lifting}
Suppose that \(\sigma\) is a reversible conical bicombing on a metric space \(X\). Then there exists a conical bicombing \(\widetilde{\sigma}\) on \(E(X)\) such that \(\widetilde{\sigma}\) and \(\sigma\) coincide on \(X\), that is, \(\widetilde{\sigma}_{xy}=\sigma_{xy}\) for all \(x\), \(y\in X\). In particular, \(X\) is a \(\sigma\)-convex subset of \(E(X)\). 
\end{Thm}

Here, a bicombing \(\sigma\) is \textit{reversible} if \(\sigma_{xy}(t)=\sigma_{yx}(1-t)\) for all \(x\), \(y\in X\) and all \(t\in [0,1]\). In \cite{MR3940917}, it is shown that every complete metric space with a conical bicombing also admits a reversible conical bicombing. Hence, Theorem~\ref{thm:main-1-neu} follows readily from Theorem~\ref{thm:lifting}. 

Theorem \ref{thm:lifting} is applicable to problems of the following form. Let \((P)\) denote a statement about conical bicombings on a metric space \(X\). Then, by Theorem \ref{thm:lifting}, if \((P)\) is true for \(E(X)\), then \((P)\) is also true for \(X\). For instance, by \cite[Theorem 1.4]{MR3940917} if \(X\) is an injective Banach space and \(\sigma\) a conical bicombing on the closed ball \(B(x, 2r)\subset X\), then on \(B(x, r)\) the bicombing \(\sigma\) is given by linear segments.  As a result, every injective Banach space admits only one conical bicombing. If \(X\) is a Banach space then \(E(X)\) admits a Banach space structure whose norm induces the metric of \(E(X)\) (see \cite[Theorem 1]{MR184061}). Hence, by Theorem \ref{thm:lifting} we obtain the following corollary:

\begin{Cor}\label{cor:uniqueBan}
A Banach space admits only one reversible conical bicombing. This unique reversible conical bicombing is given by linear segments. 
\end{Cor}
This may also be established by invoking a result of G\"ahler and Murphy (see \cite{MR642160}). As it turns out, Corollary~\ref{cor:uniqueBan} remains valid if the reversibility assumption is dropped. This is worked out in detail in Section~\ref{sec:doss}. We remark that the classical Mazur-Ulam theorem is a direct consequence of Corollary~\ref{cor:uniqueBan}. Indeed, suppose that  \(f\colon V \to W\) is a surjective isometry. The map \(\sigma\) defined by \((x,y,t)\mapsto f^{-1}\bigl( (1-t)f(x)+t f(y)\bigr)\) is a reversible conical bicombing on \(V\). Hence, by Corollary~\ref{cor:uniqueBan}, \(\sigma\) is given by linear segments, thus \(f\) is affine. 

Schechtman \cite{237871} has recently constructed a non-affine self-isometry \(f\colon C \to C\), where \(C \subset L_1 [0,1]\) is closed, convex and has empty interior. 
It follows immediately from the above argument that such a set \(C\) must necessarily admit more than one conical bicombing.
On the other hand, it follows from a theorem of Mankiewicz (see \cite{MR312214}) that any self-isometry of \(C\) is affine if the interior of \(C\) is nonempty.
This now gives rise to the natural question whether every closed convex sets whose interior is nonempty admits a unique conical bicombing (see \cite[Question 1.6]{MR3940917}). However, it turns out that already the closed upper half-plane \(H\subset \ell_\infty^2\) admits two distinct conical bicombings. This is discussed further in Example~\ref{ex:counterexample}.

\subsection{Improving conical bicombings}
It is often desirable to work with bicombings that satisfy properties which are more restrictive than \eqref{eq:conicalDEF}. 
A bicombing \(\sigma\) is said to be \textit{convex} if \(t\mapsto d(\sigma_{xy}(t), \sigma_{x'y'}(t))\)  is convex on \([0,1]\) for all \(x, y, x^\prime, y'\in X\). There are many examples of conical bicombings that are not convex (see \cite[Example 2.2]{MR3370039}). However, every consistent conical bicombing is convex. We say that a bicombing \(\sigma\) is \textit{consistent} if it is reversible and \(\sigma(x,y,s t)=\sigma(x, \sigma_{xy}(t), s)\) for all \(x\), \(y\in X\) and all \(s\), \(t\in [0,1]\). In \cite{MR1704987}, Kleiner introduced \textit{often convex} metric spaces which in our terminology are metric spaces with a consistent convex bicombing. 
We refer to \cite{chalopin2020helly}, \cite{huang2020morse}, \cite{kleiner2018higher} for some recent applications of consistent convex bicombings. 

Every Gromov hyperbolic group \(\Gamma\) acts properly and cocompactly on the proper metric space \(E(\Gamma)\), provided we endow \(\Gamma\) with the word metric with respect to any finite generating set (see \cite{MR3096307} for more details). In \cite{MR3370039}, Descombes and Lang discovered strong non-positive curvature properties of \(E(\Gamma)\). A geodesic \(\sigma\colon [0,1]\to X\) is \textit{straight} if \(t\mapsto d(\sigma(t), x)\) is convex on \([0,1]\) for all \(x\in X\). Descombes and Lang showed that \(E(\Gamma)\) has unique straight geodesics and the bicombing on \(E(\Gamma)\) given by straight geodesics is the only consistent convex bicombing on \(E(\Gamma)\). In general, it is an open question whether every proper metric space with a conical bicombing also admits a consistent convex bicombing; see \cite[p. 368]{MR3370039} and also \cite[p. 385]{MR4395714}. The following result can be regarded as a first step towards solving this difficult problem.

\begin{Thm}\label{thm:consi}
Let \(X\) be a proper metric space admitting a conical bicombing. Then there exists a consistent bicombing \(\gamma\) on \(X\) such that the following holds. Each curve \(\gamma_{xy}(\cdot)\) is a straight geodesic and \(t\mapsto d\bigl(\gamma_{xy}(t), \gamma_{x'y'}(t)\bigr)\) is convex on \([0,1]\) whenever \(d(x,y)=d(x',y')\). If \(X\) is compact or injective, then \(\gamma\) is furthermore equivariant with respect to the isometry group of \(X\).
\end{Thm}

It seems likely that the bicombing \(\gamma\) of Theorem \ref{thm:consi} is in fact convex. However, we do not know how to prove this.
A key component in the proof of Theorem \ref{thm:consi} is a sequence \((\gamma^{(n)})\) of bicombings satisfying a discrete consistency condition. 
Having \((\gamma^{(n)})\) at hand, \(\gamma\) is obtained via a straightforward ultrafilter argument. We construct \((\gamma^{(n)})\) by means of a fixed point argument on the moduli space \(\CBI(X)\) of all conical bicombings on \(X\).  The moduli space \(\CBI(X)\) is introduced and discussed in detail in Section \ref{sec:FIXED}. We hope that \(\CBI(X)\) may prove useful for further study of metric spaces with a conical bicombing.

Theorem~\ref{thm:consi} can be used to construct a visual boundary for every proper metric space admitting a conical bicombing. Let \(X\) be such a space and let \(\gamma\) denote a consistent bicombing on \(X\) satisfying the properties stated in Theorem~\ref{thm:consi}. A geodesic ray \(\xi\colon \R_+ \to X\), where \(\R_+\coloneqq [0, \infty)\), is said to be a \(\gamma\)-ray provided that \(\xi\big((1-\lambda)s+\lambda t\big)=\gamma(\xi(s),\xi(t), \lambda)\) for all \(0 \leq s \leq t\) and all \(\lambda\in [0,1]\). The visual boundary \(\partial X_\gamma\) is the set of equivalence classes of mutually asymptotic \(\gamma\)-rays. As usual, two geodesic rays \(\xi\), \(\xi^\prime\) are called asymptotic if the function \(t\mapsto d\bigl(\xi(t), \xi^\prime(t)\bigr)\) is bounded.
  
In what follows, we construct for any \(o\in X\) a natural metric \(\bar{d}_o\) on \(\overline{X}_\gamma\coloneqq X\cup \partial_\gamma X\). In Lemma~\ref{ref:unique-rep} we prove the following. For every \(o\in X\) and every \(\bar{x}\in \partial_\gamma X\), there exists a unique \(\gamma\)-ray \(\rho_{o\bar{x}}\) such that \(\rho_{o\bar{x}}(0)=o\) and \([\rho_{o\bar{x}}]=\bar{x}\). To simplify the notation, for each \(x\in X\) we define \(\rho_{o x}\colon \R_+\to X\) as follows: \(\rho_{o x}(t)=x\) for all \(t\geq d(o,x)\) and \(\rho_{o x}(t)=\gamma_{ox}(t / d(o, x))\) otherwise. In Lemma~\ref{lem:topology}, we show that 
\[
\bar{d}_o(x, x')=\int_{0}^\infty d(\rho_{ox}(t), \rho_{ox'}(t))\,e^{-t} \, dt
\]
defines a metric on \(\overline{X}_\gamma\) and the topology on \(\overline{X}_\gamma\) induced by \(\bar{d}_o\) is independent of the basepoint \(o\).

A subset \(A\) of a topological space \(X\) is called \(\mathcal{Z}\)-set if it is closed and there exists a homotopy \(h\colon X\times [0,1]\to X\) such that \(h_t(X)\subset X\setminus A\) for all \(t\in (0,1]\) and \(h_0(x)=x\) for all \(x\in X\).  For example, the boundary of a topological manifold is a \(\mathcal{Z}\)-set in that manifold.
A celebrated result of Bestvina and Mess (see \cite[Theorem~1.2]{MR1096169}) states that the Gromov closure \(\overline{P(\Gamma)}\) of an appropriately chosen Rips complex \(P(\Gamma)\) of a word hyperbolic group \(\Gamma\) has the following properties. \(\overline{P(\Gamma)}\) is an absolute retract and \(\overline{P(\Gamma)} \setminus P(\Gamma)\) is a \(\mathcal{Z}\)-set in \(\overline{P(\Gamma)}\). We have the following analogous result for \(\overline{X}_\gamma\).

\begin{Thm}\label{thm:descombes-lang-gen}
Let \(X\) be a proper metric space admitting a conical bicombing. Then \(\overline{X}_\gamma\) is an absolute retract and \(\partial_\gamma X\) is a \(\mathcal{Z}\)-set in \(\overline{X}_\gamma\).
\end{Thm}

In \cite{MR3370039}, Descombes and Lang established Theorem~\ref{thm:descombes-lang-gen} for general complete metric spaces in the case when \(\gamma\) is a consistent conical bicombing.  
To prove Theorem~\ref{thm:descombes-lang-gen} we closely follow their proof strategy, which is modeled on the boundary construction for Busemann spaces introduced in \cite{MR1425125}. 
The main difference between the proofs is that we cannot use the conical inequality \eqref{eq:conicalDEF} in our proof, since we are working with the bicombing \(\gamma\) from Theorem~\ref{thm:consi} and do not know whether \(\gamma\) is conical or not. This leads to slightly different arguments in several places.

Theorem~\ref{thm:descombes-lang-gen} has an interesting application in geometric group theory. Let \(G\) denote a group. A pair of compact topological spaces \((\overline{X}, Z)\) is called \textit{\(\mathcal{Z}\)-structure of} \(G\) if the following holds:
\begin{enumerate}
\item\label{it:1} \(\overline{X}\) is an absolute retract and \(Z\) is a \(\mathcal{Z}\)-set in \(\overline{X}\);
\item\label{it:2} \(X=\overline{X}\setminus Z\) is a proper metric space on which \(G\) acts geometrically;
\item\label{it:3} for every open cover \(\mathcal{U}\) of \(\overline{X}\) and every compact subset \(C\subset X\) all but finitely many \(G\)-translates of \(C\) are contained in some element of \(\mathcal{U}\). 
\end{enumerate}

The notion of a \(\mathcal{Z}\)-structure was coined by Bestvina in \cite{MR1381603} to formalize the notion of boundary of a group. The above definition is a generalization of Bestvina's original definition and goes back to Dranishnikov \cite{MR2216707}. The existence of a \(\mathcal{Z}\)-structure \((\overline{X}, Z)\) of a group \(G\) has many interesting consequences, since several homological invariants of \(Z\) are related to cohmological invariants of \(G\). We refer the reader to \cite{MR4007577} for a recent survey of \(\mathcal{Z}\)-structures.  
Following Farrell and Lafont (see \cite{MR2130569}), we say that a \(\mathcal{Z}\)-structure is an \(E\mathcal{Z}\)-structure if the action \(G\curvearrowright X\) can be extended to an action \(G \curvearrowright \overline{X}\) by homeomorphisms.

\begin{Cor}\label{cor:ez-structure}
Let \(G\) be a group which acts geometrically on a proper metric space \(X\) admitting a conical bicombing. Then \(G\) admits a \(\mathcal{Z}\)-structure. If \(X\) is an injective metric space, then \(G\) also admits an \(E\mathcal{Z}\)-structure.
\end{Cor}

 There is a wide variety of groups which act geometrically on proper injective metric spaces. Examples include Gromov hyperbolic groups and, more generally, Helly groups, which encompass among others weak Garside groups of finite type and Artin groups of type FC (see \cite{MR4285138} and also \cite{chalopin2020helly} for additional examples).  The fact that every Helly group admits an \(E\mathcal{Z}\)-structure has already been proved by Huang and Osajda \cite{MR4285138}.
We remark that not every group which acts geometrically on an injective metric space is necessarily a Helly group (see \cite[Corollary D]{hughes2022commensurating}).

\subsection{Acknowledgements} I am indebted to Paul Creutz, Urs Lang and Benjamin Miesch for helpful discussions. Moreover, I am indebted to the anonymous reviewer for several suggestions which improved the exposition of the article. Parts of this work are contained in the author's PhD thesis \cite{20.500.11850/398970}. 

%%%%%%%%%%%%%%%%%%%%%%%%%%%%%%%%%%%%%%%%%%%%%%%%%%%%%%%%%%%%%%%%%%%
\section{\(1\)-Wasserstein distances and barycentric metric spaces}\label{sec:section-two}
%%%%%%%%%%%%%%%%%%%%%%%%%%%%%%%%%%%%%%%%%%%%%%%%%%%%%%%%%%%%%%%%%%%

\subsection{The \(1\)-Wasserstein distance}\label{sec:wasser} We recall the basic properties of the \(1\)-Wasserstein distance. 
Let \(X\) be a metric space and let \(P(X)\) denote the set of all Radon probability measures on \(X\).
For \(\mu\), \(\nu\in P(X)\) we introduce the \textit{\(1\)-Wasserstein distance}
\[
W_1(\mu, \nu)\coloneqq\inf_{\pi} \int_{X\times X} d(x,y) \,d\pi(x,y) \quad \quad (\mu, \nu\in P(X)),
\]
where the infimum is taken over all couplings of the pair \((\mu, \nu)\). Here,
\(\pi\in P(X\times X)\) (we equip \(X\times X\) with the 1-product metric) is a \textit{coupling} 
of  \((\mu, \nu)\) if \(\pi(B\times X)=\mu(B)\) and \(\pi(X\times B)=\nu(B)\) for all Borel subsets \(B\subset X\).
Let \(P_1(X)\) denote the set of all \(\mu\in P(X)\) such that \(W_1(\mu, \delta_{x_0}) <\infty\) for some  \(x_0\in X\). 
The celebrated Kantorovich-Rubinstein duality theorem states that 
\begin{equation}\label{eq:kanto}
W_1(\mu, \nu)=\sup\biggl\{ \int_X f\, d\mu-\int_X f \, d\nu \,: \,f \in \Lip_1(X) \biggr\}
\end{equation}
for all \(\mu, \nu\in P_1(X)\). We use \(\Lip_1(X)\) to denote the set of all \(1\)-Lispchitz functions \(f\colon X\to \R\). 
We remark that if the supports of \(\mu\) and \(\nu\) are finite, then \eqref{eq:kanto} follows easily from the strong duality theorem of linear programming. For a thorough discussion of the Kantorovich–Rubinstein theorem we refer the reader to the excellent survey article \cite{edwards}.

As a direct consequence of \eqref{eq:kanto}, the pair \((P_1(X), W_1)\) is a metric space.
Moreover, for every \(L\)-Lipschitz map \(f\colon X\to Y\) the push-forward map \(f_\#\colon P_1(X)\to P_1(Y)\) is \(L\)-Lipschitz as well; see \cite[Lemma 2.1]{MR2480975}. 
\begin{Lem}\label{lem:isoW}
Let \(X\) and \(Y\) denote metric spaces. If \(i\colon X\to Y\) is an isometric embedding, then \(i_\#\colon P_1(X)\to P_1(Y)\) is an isometric embedding as well.
\end{Lem}
\begin{proof}
It suffices to prove that the map \(\Lip_1(Y)\to \Lip_1(X)\) defined by \(g\mapsto g\circ i\) is surjective. To this end, let \(f\in \Lip_1(X)\) and let \(g\colon Y \to \R \) be defined by
\begin{equation*}
y\mapsto \inf_{x\in X} \,\bigl[ f(x)+d(y, i(x))\bigr].
\end{equation*}
We remark that such functions \(g\) occur naturally in the context of the McShane extension theorem (see \cite[Remark 2.4]{MR4103879} for more information). Notice that \(f(x')+d(i(x), i(x'))\geq f(x')+\abs{f(x)-f(x')}\geq f(x)\) for all \(x\), \(x'\in X\). Consequently, since \(f(x)\geq g(i(x))\), it follows that \(f=g\circ i\). In addition,  for all \(y\), \(y'\in Y\), 
\begin{align*}
\abs{ g(y) -g(y')}&=\abs{ \inf_{x\in X} \bigl[ f(x)+d(y, i(x))\bigr]-\inf_{x\in X} \bigl[ f(x)+d(y', i(x))\bigr]} \\
&\leq\sup_{x\in X}\, \abs{ d(i(x),y)-d(i(x), y')} \leq d(y,y').
\end{align*}
Hence, \(g\) is a \(1\)-Lipschitz function on \(Y\) such that \(f=g\circ i\), as desired. 
\end{proof}
 
In this article, we will mainly work with measures which are supported at finitely many points. For such measures, the following formula for the \(1\)-Wasserstein distance is well-known. 

\begin{Prop}\label{prop:formula}
Assume that \(x_1, y_1, \dots, x_n, y_n\in X\) are (not necessarily distinct) points of a metric space \(X\). Then
\[W_1\Big(\frac{1}{n}\sum_{i=1}^n \delta_{x_i},\frac{1}{n}\sum_{i=1}^n \delta_{y_i}\Big)=\frac{1}{n} \,\min_{\pi\in S_n} \sum_{i=1}^n d(x_i, y_{\pi(i)}),\]
where \(S_n\) denotes the symmetric group of degree \(n\). 
\end{Prop}
\begin{proof}
We sketch the proof indicated in \cite[p. 5]{MR1964483}. Another proof using Hall’s marriage theorem can be found in \cite[p. 953]{MR3563261}. We abbreviate \(\mu\coloneqq\frac{1}{n}\sum_{i=1}^n \delta_{x_i}\) and \(\nu\coloneqq\frac{1}{n}\sum_{i=1}^n \delta_{y_i}\). 
Clearly, 
\[ W_1(\mu, \nu)=\min\Big\{ \frac{1}{n}\sum_{i,j=1}^n p_{ij} d(x_i, y_j) : P=(p_{ij}) \textrm{ is doubly stochastic}\Big\}. \]
A non-negative \(n\times n\) matrix \(P\) is doubly stochastic if \(Pj=P^tj=j\) for the all-ones vector \(j\in \R^n\). 
The Birkhoff–von Neumann theorem states that each doubly stochastic matrix is equal to a finite convex combination of permutation matrices. Hence, by the above, 
\[
W_1(\mu, \nu)=\min\Big\{ \frac{1}{n}\sum_{i=1}^n \delta_{i, \pi(i)}d(x_i, y_j) : \pi\in S_n \Big\},
\]
as desired.
\end{proof}

Our next lemma computes \(W_1(\mu, \nu)\) in the special case when the supports of \(\mu\) and \(\nu\) consist of at most two points. The proof is straightforward and follows from solving a certain system of linear equations. Alternatively, we could also invoke Proposition \ref{prop:formula} and a simple limit argument. In the following, we use the notation \(a\vee b \coloneqq\max\{ a, b\}\) and \(a \wedge b\coloneqq\min\{ a, b\}\). 

\begin{Lem}\label{lem:wasser} Let \(x_1\), \(x_2\), \(y_1\), \(y_2\in X\) and \(s\), \(t\in [0,1]\). Then 
\begin{align*}
&\hspace{-4em}W_1\big((1-s)\delta_{x_1}+s\delta_{x_2}, (1-t)\delta_{y_1}+t\delta_{y_2}\big) \\
=\min_{\lambda\in I_{s,t}}\Bigl[&(1-(s+t)+\lambda)d(x_1,y_1)+(s-\lambda)d(x_2, y_1) +(t-\lambda)d(x_1, y_2)+\lambda d(x_2,y_2)\Bigr],
\end{align*}
where \(I_{s,t}\coloneqq\bigl[0\vee (s+t-1), s\wedge t\bigr]\).
\end{Lem}

\begin{proof}
We abbreviate \(\mu\coloneqq(1-s)\delta_{x_1}+s\delta_{x_2}\) and \(\nu\coloneqq(1-t)\delta_{y_1}+t\delta_{y_2}\). Notice that \(\pi \in P_1(X\times X)\) is a coupling of \((\mu, \nu)\) if and only if \(\pi=\sum_{i,j} \pi_{ij}\delta_{(x_i,y_j)}\) with \(0\leq \pi_{ij}\leq 1\) and
\begin{equation*}
\pi_{11}+\pi_{12}=1-s, \quad \pi_{21}+\pi_{22}=s, \quad \pi_{11}+\pi_{21}=1-t, \quad \pi_{12}+\pi_{22}=t.  
\end{equation*}
The solution set of this system of linear equations equals
\(v(s,t)+\big\{(\lambda,\, -\lambda, \, -\lambda, \, \lambda) : \lambda\in \R \big\},\) where \(v(s,t)\coloneqq \bigl(1-(s+t), \, t,\, s, \,0\bigr)\).
Since \(0 \leq \pi_{ij} \leq 1\), letting \(I_{s,t}\coloneqq\bigl[0\vee (s+t-1), s\wedge t\bigr]\) we get that \(W_1(\mu, \nu)\) is equal to
\begin{align*}
\min_{\lambda\in I_{s,t}}\Bigl[&(1-(s+t)+\lambda)d(x_1,y_1)+(t-\lambda)d(x_1,y_2)+(s-\lambda)d(y_1,x_2)+\lambda d(x_2,y_2)\Bigr],
\end{align*}
as was to be shown. 
\end{proof}

%%%%%%%%%%%%%%%%%%%%%%%%%%%%%%%%%%%%%%%%%%%%%%%%%%%%%%%%%%%%%%%%%%%
\subsection{Barycentric metric spaces}\label{sec:Contracting}
%%%%%%%%%%%%%%%%%%%%%%%%%%%%%%%%%%%%%%%%%%%%%%%%%%%%%%%%%%%%%%%%%%%

In what follows, we introduce barycentric metric spaces and recall their close connection to conical bicombings. 
The following definition is due to Sturm (see \cite[Remark 6.4]{MR2039961}).

\begin{Def}
Let \(X\) denote a metric space. A \(1\)-Lipschitz map \(\beta\colon P_1(X)\to X\) is a \textit{contracting barycenter map} if \(\beta(\delta_x)=x\) for all \(x\in X\). A metric space is said to be a \textit{barycentric metric space} if it admits a contracting barycenter map. 
\end{Def} 
There are many examples of barycentric metric spaces. 
In particular, every injective metric space is barycentric. This can be seen by considering the isometric embedding \(X\to P_1(X)\) defined by \(x\mapsto \delta_x\). Moreover, every Banach space admits a unique contracting barycenter map. Indeed, one can show that if \(E\) denotes a real Banach space, then  \(\beta\colon P_1(E)\to E\) defined by
\[
\beta(\mu)\coloneqq \int_E x\, \mu(dx) ,
\]
where the integral on the right hand side is the strong Bochner integral, is the only contracting barycenter map on \(E\) (see \cite[Proposition 3.6]{MR3819996}). It is well-known that the Cartan barycenter map on a complete \(\CAT(0)\) spaces is contracting (see \cite[Theorem 6.3]{MR2039961} or \cite[Lemma 4.2]{MR1810752}), and so every complete \(\CAT(0)\) space is barycentric. More generally, Navas \cite{MR3035300} established that in fact every complete Busemann space is a barycentric metric space.  

In the following lemma we show by standard arguments that every barycentric metric space admits a conical bicombing.

\begin{Lem}\label{lem:trivial1}
Suppose that \(\beta\colon P_1(X)\to X\) is a contracting barycenter map on a metric space \(X\). Then \(\sigma_\beta\colon X\times X \times [0,1]\to X\) defined by 
\[(x,y,t)\mapsto \beta\bigl( (1-t)\delta_x+t \delta_y \bigr)\] 
is a reversible conical  bicombing on \(X\).
\end{Lem}
\begin{proof}
Fix \(x\), \(y\in X\) and \(s\), \(t\in [0,1]\) such that \(s\leq t\). We abbreviate \(\sigma\coloneqq\sigma_\beta\). Using that \(\beta\) is a contracting barycenter map, we obtain 
\begin{equation*}
\begin{split}
d(\sigma_{xy}(s), \sigma_{xy}(t)) \leq W_1\bigl((1-s)\delta_x+s\delta_y, (1-t)\delta_x+t\delta_y\bigr)=(t-s)\, d(x,y),
\end{split}
\end{equation*}
where the equality is due to Lemma \ref{lem:wasser}. Since
\[
d(x,y)\leq d(x, \sigma_{xy}(s))+d(\sigma_{xy}(s), \sigma_{xy}(t))+d(\sigma_{xy}(t),y) \leq d(x,y),
\]
it follows that \(d(\sigma_{xy}(s), \sigma_{xy}(t))=(t-s) d(x,y)\), and so \(\sigma_{xy}\) is a geodesic from \(x\) to \(y\). Next, we prove \eqref{eq:conicalDEF}. Let \(t\in [0,1]\). Using Lemma \ref{lem:wasser}, we obtain 
\begin{equation}\label{eq:onical-simple}
d(\sigma_{xy}(t), \sigma_{xy}(t))\leq W_1\big((1-t)\delta_{x}+t \delta_{y}, (1-t)\delta_{x}+t \delta_{z} \big) = t\, d(y, z)
\end{equation}
for all \(x\), \(y\), \(z\in X\). Since \(\sigma\) is reversible, 
\[
d(\sigma_{xy}(t), \sigma_{x'y'}(t))\leq d(\sigma_{xy}(t), \sigma_{xy'}(t))+d(\sigma_{y'x}(1-t), \sigma_{y' x'}(1-t)),
\] 
and thus by using \eqref{eq:onical-simple}, we obtain \eqref{eq:conicalDEF}, as desired.
\end{proof}

Conversely, every complete metric space with a conical bicombing is barycentric:

\begin{Thm}\label{thm:equivBar}
Let \(X\) denote a complete metric space. Then the following statements are equivalent:
\begin{enumerate}
\item \(X\) admits a conical bicombing.
\item \(X\) is a barycentric metric space.
\end{enumerate}
\end{Thm}

Theorem \ref{thm:equivBar} is essentially known. The key idea leading to Theorem \ref{thm:equivBar} is a
1-Lipschitz barycenter construction first described by Es-Sahib and Heinich \cite{MR1768010}. This barycenter construction has been improved by Navas in \cite{MR3035300}. A streamlined proof of Navas's construction using elementary statistics can be found in \cite{MR3563261}. The implication \((2.) \Longrightarrow (1.)\) is a direct consequence of Lemma~\ref{lem:trivial1}. The other direction follows from the following result, which is essentially due to Descombes (see \cite{MR3563261}). 

\begin{Thm}\label{thm:contBary}
Let \(X\) denote a complete metric space admitting a conical bicombing \(\sigma\). Then there exists a contracting barycenter map \(\beta_\sigma \colon P_1(X)\to X\) such that \(\beta_\sigma(\mu)\in \overline{\conv}_\sigma(\spt(\mu))\) for all \(\mu \in P_1(X)\). 
\end{Thm} 

The support \(\spt(\mu)\) is the set of all points \(x\in X\) such that \(\mu(U) > 0\) for all open subsets \(U\subset X\) containing \(x\).
For \(A\subset X\) the \textit{closed \(\sigma\)-convex hull} of \(A\), denoted by \(\overline{\textrm{conv}}_{\sigma}(A)\), is the closure of the smallest \(\sigma\)-convex set that contains \(A\). 

\begin{proof}[Proof of Theorem \ref{thm:contBary}]
We give only the main ideas of the proof. By virtue of \cite[Proposition 1.3]{MR3940917}, we obtain a reversible conical bicombing \(\tau\) on \(X\) such that \(\overline{\conv}_{\tau}(A)\subset \overline{\conv}_{\sigma}(A)\) for all \(A\subset X\). We set \(b_1(x)\coloneqq x\) and \(b_2(x,y)\coloneqq\tau_{xy}(\tfrac{1}{2})\) for all \(x,y\in X\). Using \cite[Proposition 3.4]{MR3819996}, we obtain a sequence of maps \((b_n\colon X^n \to X)_{n\geq 3}\) satisfying
\begin{equation}\label{eq:1}
d(b_n(x), b_n(y))\leq \frac{1}{n}\, \min_{\pi\in S_n}\sum_{i=1}^n d(x_i, y_{\pi(i)})
\end{equation}
for all \(n\geq 3\) and all \(x,y\in X^n\).
Given \(x\in X^n\),  for each \(k\geq 1\) we write \(Q^k(x)\in X^{kn}\) to denote \((x,\ldots, x)\). Descombes \cite[Theorem 2.5 (1)]{MR3563261} proved that the limit
\[b(x)\coloneqq\lim_{k\to+\infty} b_{nk}(Q^{k}(x))\]
exists for all \(x\in X^n\). Moreover, if \(x=(x_1, \ldots, x_n)\), then 
\begin{equation}\label{eq:2}
b(x)\in\overline{\conv}_{\tau}(\{x_1, \ldots, x_n\}).
\end{equation} 
By \eqref{eq:1}, \eqref{eq:2}, and Proposition \ref{prop:formula}, the map \(\beta\colon P_\Q(X)\to X\) given by
\[\mu=\frac{1}{n}\bigl( \delta_{x_1}+\dotsm+\delta_{x_n}\bigr)\mapsto \beta(\mu)\coloneqq b(x_1, \ldots, x_n)\]
is well-defined, \(\beta(\mu)\in \overline{\conv}_\tau(\spt(\mu))\), and \(d(\beta(\mu), \beta(\nu))\leq W_1(\mu, \nu)\) for all \(\mu, \nu \in P_{\Q}(X)\), where \(P_\Q(X)\subset P_1(X)\) denotes the set of all Radon probability measures on \(X\) with finite support and rational weights. The map \(\beta\colon P_\Q(X)\to X\) extends to a contracting barycenter map \(\beta\) on \(X\), for \(X\) is complete and \(P_{\Q}(X)\) is \(W_1\)-dense in \(P_1(X)\) (see \cite[Proposition 3.2]{MR3819996}). Now, it is easy to check that \(\beta_\sigma\coloneqq\beta\) has the desired properties. The theorem follows. 
\end{proof}

We remark that in view of Theorem~\ref{thm:lifting}, to prove Theorem~\ref{thm:contBary} it would suffice to consider the special case when \(X\) is an injective metric space. However, to prove this special case seems to be as difficult as proving the general case. 

%%%%%%%%%%%%%%%%%%%%%%%%%%%%%%%%%%%%%%%%%%%%%%%%%%%%%%%%%%%%%%%%%%%%%%%%%%%% 
\section{Extending conical bicombings}\label{sec:ladst}
%%%%%%%%%%%%%%%%%%%%%%%%%%%%%%%%%%%%%%%%%%%%%%%%%%%%%%%%%%%%%%%%%%%%%%%%%%%%

\subsection{Consequences of the conical inequality} The following lemma shows that every reversible conical bicombing satisfies an inequality which is slightly stronger than \eqref{eq:conicalDEF}.

\begin{Lem}\label{lem:slightlyStronger}
Let \(X\) be a metric space, \(A\subset X\),  and \(\{\sigma_{xy}(\cdot) : x, y\in A\}\) a collection of geodesics \(\sigma_{xy}\colon [0,1]\to X\) such that \(\sigma_{xy}(0)=x, \,\sigma_{xy}(1)=y\), and \(\sigma_{xy}(t)=\sigma_{yx}(1-t)\) for all \(t\in [0,1]\) and all \(x,y\in A\). If
\begin{equation}\label{eq:test8} 
d(\sigma_{xy}(t), \sigma_{xz}(t)) \leq t\,d(y,z) 
\end{equation}
for all \(x,y,z\in A\) and all \(t\in [0,1]\), then
\[d(\sigma_{x_1 x_2}(t), \sigma_{y_1 y_2}(t)) \leq W_1\bigl( (1-t)\delta_{x_1}+t \delta_{x_2} , (1-t)\delta_{y_1}+t \delta_{y_2}\bigr)\] 
for all \(t\in [0,1]\) and all \(x_1, x_2, y_1, y_2\in A\).
\end{Lem}
\begin{proof}
Without loss of generality, we may suppose that \(t\in [1/2,1]\). We retain the notation from Lemma \ref{lem:wasser}.
For \(s=t\), one has \(I_{s,t}=[2t-1, t]\). Thus, by substituting \(\varepsilon\coloneqq t-\lambda\) in Lemma \ref{lem:wasser}, we obtain
\begin{align}\label{eq:express}
W_1\big( (1-t)\delta_{x_1}&+t \delta_{x_2}, (1-t)\delta_{y_1}+t \delta_{y_2}\big)  \\
=\min_{\varepsilon\in [0,1-t]}\Bigl[ &\varepsilon(d(x_1,y_2)+d(y_1,x_2)) + (t-\varepsilon)d(x_2, y_2)+((1-t)-\varepsilon)d(x_1, y_1) \Bigr]. \nonumber
\end{align}
On the one hand, we compute
\begin{align*}
d(\sigma_{x_1 x_2}(t), \sigma_{y_1 y_2}(t)) &\leq d(\sigma_{x_1 x_2}(t), \sigma_{x_1 y_2}(t)) \\
&\quad+d(\sigma_{y_2 x_1}(1-t), \sigma_{y_2 y_1}(1-t)), 
\end{align*}
and so, by the use of \eqref{eq:test8}, we get
\begin{equation}\label{eq:one}
d(\sigma_{x_1 x_2}(t), \sigma_{y_1 y_2}(t)) \leq (1-t) d(x_1, y_1)+t d(x_2, y_2),
\end{equation} 
but on the other hand, 
\begin{align*}
d(\sigma_{x_1 x_2}(t), \sigma_{y_1 y_2}(t)) &\leq d(\sigma_{x_2 x_1}(1-t), \sigma_{x_2 y_2}(1-t)) \\
&\quad+d(\sigma_{x_2 y_2}(1-t), \sigma_{x_2 y_2}(t)) +d(\sigma_{x_2 y_2}(t), \sigma_{y_1 y_2}(t)) 
\end{align*}
and therefore
\begin{align}\label{eq:two}
d(\sigma_{x_1 x_2}(t), \sigma_{y_1 y_2}(t)) &\leq (1-t) d(x_1,y_2)\\
&\quad+(2t-1) d(x_2, y_2)+(1-t)d(x_2, y_1). \nonumber
\end{align}
By combining \eqref{eq:express} with \eqref{eq:one} and \eqref{eq:two}, we find that
\[
d(\sigma_{x_1 x_2}(t), \sigma_{y_1 y_2}(t)) \leq W_1\big( (1-t)\delta_{x_1}+t \delta_{x_2} , (1-t)\delta_{y_1}+t \delta_{y_2}\big),
\]
as desired.
\end{proof}
Lemma \ref{lem:slightlyStronger} tells us that if \(\sigma\) is a reversible conical  bicombing on \(X\), then the map
\((1-t)\delta_x+t\delta_y\mapsto \sigma_{xy}(t)\) is \(1\)-Lipschitz with respect to the \(1\)-Wasserstein distance. This observation is the key idea behind the proof of Theorem \ref{thm:lifting}. 

\subsection{Partially defined barycenter maps}

In the following we will prove that any partially defined barycenter map \(\beta\colon M\to X\), where \(M\subset P_1(X)\), can be extended to a contracting barycenter map \(\widetilde{\beta}\colon P_1(E(X))\to E(X)\). The proof crucially relies on the following well-known property of the injective hull.
\begin{Lem}\label{lem:uniqueness-1}
Let \(X\) denote a metric space and \((E(X), i)\) its injective hull. If \(z\), \(z'\in E(X)\) satisfy \(d(z,i(x))=d(z',i(x))\) for all \(x\in X\), then \(z=z'\). 
\end{Lem}
\begin{proof}
Consider the metric space 
\[
\Delta_1(X)\coloneqq \Bigl\{ f\in \Lip_1(X) : f(x)+f(x') \geq d(x,x') \text{ for all }x,x'\in X \Bigr\}
\]
equipped with the supremum metric \(d_\infty\), that is, \(d_\infty(f,g)\coloneqq\norm{f-g}_\infty=\sup_{x\in X} \abs{f(x)-g(x)}\) for all \(f\), \(g\in \Delta_1(X)\). It is straightforward to show that \(\Delta_1(X)\) is an injective metric space (see, for example, \cite[Proposition 3.2]{MR3096307}). Moreover, for any \(f\in \Delta_1(X)\), one has that \(\norm{f-d_x}_\infty=f(x)\) for all \(x\in X\), where \(d_x\in \Delta_1(X)\) denotes the distance function from \(x\). Clearly, \(j\colon X\to \Delta_1(X)\) defined by \(x\mapsto d_x\) is an isometric embedding. Hence, by the definition of the injective hull, there exists an isometric embedding \(k\colon E(X) \to \Delta_1(X)\) such that \(i(x)\mapsto d_x\). By construction, \(k(z)(x)=\norm{k(z)-k(i(x))}_\infty=d(z,i(x))\) for all \(x\in X\). Therefore, by our assumptions on \(z\) and \(z'\), it follows that \(k(z)=k(z')\), but this is only possible if \(z=z'\), for \(k\) is an isometric embedding. This completes the proof.  
\end{proof}

We say that \(\beta \colon M\to X\) is a \textit{partially defined barycenter map} if \(M\subset P_1(X)\) contains \(\{\delta_x : x\in X\}\) and \(\beta\) is \(1\)-Lipschitz. 

\begin{Lem}\label{prop:divergent} Let \(X\) be a metric space and denote by \((E(X), i)\) its injective hull. Then for any  partially defined barycenter map \(\beta\colon M\to X\) there exists a contracting barycenter map \(\widetilde{\beta}\colon P_1(E(X))\to E(X)\) which extends \(\beta\), that is, \(\widetilde{\beta} (i_\#(\mu))=i(\beta(\mu))\) for all \(\mu\in M\).
\end{Lem}

\begin{proof}
The composition \(i\circ \beta\) is a \(1\)-Lipschitz map and the push-forward map \(i_\#\colon P_1(X)\to P_1(E(X))\) is an isometric embedding; see Lemma \ref{lem:isoW}. Therefore, as \(E(X)\) is an injective metric space, there exists a \(1\)-Lipschitz map \(\widetilde{\beta}\colon P_1(E(X))\to E(X)\) such that \(i(\beta(\mu))=\widetilde{\beta} (i_\#(\mu))\) for all \(\mu\in M\).

To finish the proof it remains to show that \(\widetilde{\beta}(\delta_z)=z\) for all \(z\in E(X)\). To this end, let \(f\colon E(X)\to E(X)\) be defined by \(z\mapsto \widetilde{\beta}(\delta_z)\). By construction, \(f\) is \(1\)-Lipschitz and \(f(i(x))=i(x)\) for all \(x\in X\) and thus \(f\circ i\) is an isometric embedding. Consequently, by the definition of the injective hull, \(f\) is an isometric embedding as well. In particular, \(d(f(z), i(x))=d(z, i(x))\) for all \(x\in X\). Hence, Lemma~\ref{lem:uniqueness-1} implies that \(f(z)=z\) for all \(z\in E(X)\). Since \(f(z)=\widetilde{\beta}(\delta_z)\) this gives the desired result.
\end{proof}

\subsection{Extensions to the injective hull}
Next, we prove Theorem~\ref{thm:lifting} from the introduction, which states that any reversible conical bicombing on a metric space \(X\) can be extended to a concial bicombing on \(E(X)\). The proof is a straightforward application of Lemmas \ref{lem:trivial1}, \ref{lem:slightlyStronger} and \ref{prop:divergent}.

\begin{proof}[Proof of Theorem \ref{thm:lifting}]
We put 
\[
M\coloneqq \bigl\{ (1-t)\delta_x+t\delta_y \colon x,y\in X,\, t\in [0,1]\bigr\}.
\]
Due to Lemma \ref{lem:slightlyStronger}, it follows that \(\beta\colon M\to X\) defined by \((1-t)\delta_x+t\delta_y\mapsto \sigma_{xy}(t)\) is 1-Lipschitz and thus it is a partially defined barycenter map. Therefore, by virtue of Lemma \ref{prop:divergent} there exists a contracting barycenter map \(\widetilde{\beta}\colon P_1(E(X))\to E(X)\) such that \(\widetilde{\beta}(i_\#(\mu))=i(\beta(\mu))\) for all \(\mu\in M\). Hence, \(\widetilde{\sigma}\colon E(X)\times E(X)\times [0,1]\to E(X)\) defined by \((x,y,t)\mapsto \widetilde{\beta}\bigl( (1-t)\delta_x+t \delta_y \bigr)\)
is a reversible conical bicombing; see Lemma \ref{lem:trivial1}. By construction, \(\widetilde{\sigma}_{i(x)i(y)}=i\circ\sigma_{xy}\) for all \(x\), \(y\in X\).  
\end{proof}

\subsection{Doss expectation}\label{sec:doss}
In what follows, we prove the following generalization of Corollary \ref{cor:uniqueBan}.
\begin{Prop}\label{prop:Unique}
Any normed real vector space admits only one conical bicombing. This unique conical bicombing is given by linear segments. 
\end{Prop}
To establish Proposition \ref{prop:Unique} we consider the Doss expectation of a measure.
Let \(X\) be a metric space. For each \(\mu\in P_1(X)\) the set
\[\mathbb{E}_D[\mu]\coloneqq\bigl\{ z\in X : d(z,x)\leq W_1(\mu, \delta_x) \textrm{ for all } x\in X \bigr\}\]
is called \textit{Doss expectation} of \(\mu\). See \cite[Section 2.3.]{MR2132405} for other notions of expectation in metric spaces. Notice that 
\(\sigma_{xy}(t)\in \mathbb{E}_D[(1-t)\delta_x +t \delta_y]\)
for all \(x\), \(y\in X\) and all \(t\in [0,1]\) whenever \(\sigma\) is a conical bicombing on \(X\). Conversely, if \(X\) is an injective metric space, then the map \(\sigma\mapsto \sigma_{xy}(t)\in \mathbb{E}_D[(1-t)\delta_x +t \delta_y]\) is surjective. 

\begin{Lem}\label{lem:Doss1}
Let \(X\) be an injective metric space. Then for all \(x\), \(y\in X\) and all \(t\in [0,1]\) the following holds. For every \(z\in \mathbb{E}_D[(1-t)\delta_x +t \delta_y]\) there exists a reversible conical bicombing \(\sigma\) on \(X\) such that \(\sigma_{xy}(t)=z\). In particular, if \(X\) admits only one reversible conical bicombing, then \(\mathbb{E}_D[(1-t)\delta_x +t \delta_y]\) is a singleton. 
\end{Lem}

\begin{proof}
Fix \(z\in\mathbb{E}_D[(1-t)\delta_x +t \delta_y]\), abbreviate \(\mu\coloneqq(1-t)\delta_x +t \delta_y\) and put \(M\coloneqq\{ \delta_x : x\in X\}\cup \{\mu\}\). The map \(f\colon M\to X\) defined by \(\delta_x\mapsto x\) and \(\mu\mapsto z\) is a partially defined barycenter map. Thus, as \(X\) is injective, there exists a contracting barycenter map \(\beta\colon P_1(X)\to X\) such that \(\beta(\mu)=f(\mu)\) for all \(\mu\in M\). Let \(\sigma_\beta\) be defined as in Lemma \ref{lem:trivial1}. 
By construction, \(\sigma_\beta(x,y,t)=z\). Since \(\sigma_\beta\) is a reversible conical bicombing, the lemma follows. 
\end{proof}

We proceed by proving Proposition \ref{prop:Unique}. 
\begin{proof}[Proof of Proposition \ref{prop:Unique}]
Let \(V\) be a normed vector space over \(\R\). It suffices to show that \(\mathbb{E}_D[(1-t)\delta_x+t \delta_y]\) is a singleton for all \((x,y,t)\in V\times V\times [0,1]\). Let \((E(V), i)\) denote the injective hull of \(V\) and fix \((x,y,t)\in V\times V \times [0,1]\). By the use of Lemma \ref{prop:divergent}, it is not hard to check that
\begin{equation}\label{eq:cont1}
i(\mathbb{E}_D[(1-t)\delta_x+t \delta_y])\subset \mathbb{E}_D[(1-t)\delta_{i(x)}+t \delta_{i(y)}].
\end{equation}
Since \(V\) is a normed real vector space, a result due to Isbell \cite[Theorem 1]{MR184061} (see also \cite[Theorem 2.1]{MR1177160}), tells us that there exists a Banach space structure on \(E(V)\) such that its norm induces the metric of \(E(V)\). Hence, from Corollary \ref{cor:uniqueBan} and Lemma \ref{lem:Doss1}, it follows that \(\mathbb{E}_D[(1-t)\delta_{i(x)}+t \delta_{i(y)}]\) is a singleton. By \eqref{eq:cont1}, \(\mathbb{E}_D[(1-t)\delta_x+t \delta_y]\) is a singleton as well, as desired. 
\end{proof}

It seems natural to ask if Proposition~\ref{prop:Unique} can be generalized. For example, one may ask if any closed convex subset of a Banach space admits a unique conical bicombing. However, we show in the following example that already certain convex subsets of \(\ell_2^\infty\) admit two distinct conical bicombings (and thus infinitely many).
This gives a negative answer to Question~1.6 of \cite{MR3940917}. 
\begin{Ex}\label{ex:counterexample}
We consider the Banach space  \(\ell_\infty^2\coloneqq(\R^2, \norm{\cdot}_{\infty})\), where \(\norm{\cdot}_{\infty}\) denotes the supremum norm. 
We put \(H\coloneqq\{ (s,t)\in \R^2 : t \geq 0\}\subset \ell_\infty^2.\)
In what follows, we show that \(H\) admits two distinct conical bicombings. Let \(\pi_i\colon H\to \R\), for \(i=1,2\), denote the projection onto the \(i\)th coordinate axis. A straightforward computation shows that 
\(\beta\colon P_1(H)\to H\) defined by \(\mu\mapsto (\beta_1(\mu), \beta_2(\mu))\), where
\[\beta_i(\mu)\coloneqq\inf_{p\in H} \bigl( \pi_i(p)+W_1(\delta_p, \mu)\bigr),\]
is a contracting barycenter map.
We set \(p_1\coloneqq(-1,0), \, p_2\coloneqq(1,0)\), and \(\mu\coloneqq\tfrac{1}{2}\delta_{p_1}+\tfrac{1}{2}\delta_{p_2}\). We claim that \(\beta(\mu)=(0,1)\). Clearly,
\[d_\mu\coloneqq\min_{p\in H} W_1(\delta_p, \mu) \leq \beta_2(\mu).\] 
Notice that if \(p\in H\) satisfies \(W_1(p, \mu)=d_\mu\), then \(W_1(r(p), \mu)=d_\mu\), where \(r\colon H\to H\) is the reflection about the \(y\)-axis. Thus, every point \(q\) on the linear segment \([p, r(p)]\) satisfies \(W_1(\delta_q, \mu)=d_\mu\). Hence, there exists \(u\in H\) such that \(\pi_1(u)=0\) and \(W_1(\delta_{u}, \mu)=d_\mu\). Consequently, \(1\leq W_1(\delta_{u}, \mu)=d_\mu \leq \beta_2(\mu).\)
Since \[\norm{\beta(\mu)-p_1}_{\infty}=\norm{\beta(\mu)-p_2}_{\infty}=\frac{1}{2}\norm{p_1-p_2}_{\infty}\text{,}\] 
we obtain \(\beta_2(\mu)=1\) and thus \(\beta(\mu)=(0,1)\), as claimed. 
The map \(\sigma\colon H\times H\times [0,1]\to H\) defined by \((p,q,t)\mapsto \beta\bigl((1-t)\delta_p+t \delta_q\bigr)\) is a reversible conical bicombing on \(H\); see Lemma \ref{lem:trivial1}.
Let \(\lambda\) denote the conical bicombing on \(H\) given by linear segments. By construction, \(\sigma(p_1, p_2, \tfrac{1}{2})=(0,1)\) and \(\lambda(p_1, p_2, \tfrac{1}{2})=(0,0)\). Hence, we infer \(\sigma\neq \lambda\) and thus \(H\) admits two distinct conical bicombings, as desired. 
\end{Ex}

%%%%%%%%%%%%%%%%%%%%%%%%%%%%%%%%%%%%%%%%%%%%%%%%%%%%%%%%%%%%%%%%%%%%%%%%%%%%%%%%%%%
\section{Conical  bicombings as fixed points}\label{sec:FIXED}
%%%%%%%%%%%%%%%%%%%%%%%%%%%%%%%%%%%%%%%%%%%%%%%%%%%%%%%%%%%%%%%%%%%%%%%%%%%%%%%%%%%

\subsection{Conical  bicombings on \(\CBI(X)\)} Let \(X\) be a metric space and let \(\CBI(X)\) denote the set of all conical  bicombings on \(X\). In the following, we show that \(\CBI(X)\) can endowed with a metric such that the resulting metric space admits a conical bicombing whenever \(X\) does. 
Given \(o\in X\) let \(D_{o}\colon\CBI(X) \times\CBI(X) \to \R\) be defined by 
\begin{equation*}
D_o(\sigma, \tau)\coloneqq \sup\Bigl\{  \, 3^{-k} d(\sigma_{xy}(t), \tau_{xy}(t))  : k\geq 0, \, x,y\in B_{2^k}(o), \, t\in [0,1] \Bigr\}.
\end{equation*}
Clearly, \(D_o\) is a metric on \(\CBI(X)\). We have defined \(D_o\) in such a way that for proper metric spaces \(X\)  the induced topology on \(\CBI(X)\) coincides with the topology \(\mathcal{T}_K\) of uniform convergence on compact sets; see Lemma~\ref{lem:compact}. This will be important in Section~\ref{sec:constr}, where fixed point arguments on \(\CBI(X)\) are employed to construct bicombings which satisfy certain consistency conditions.

The following lemma shows that each conical bicombing on \(X\) induces a conical bicombing on \((\CBI(X), D_{o})\).

\begin{Lem}\label{lem:C-biBicombing}
Let \(X\) be a metric space and fix \(o\in X\). If \(\phi\) is a conical bicombing on \(X\), then for all \(\sigma\), \(\tau\in \CBI(X)\), the map 
\(\Phi_{\sigma\tau}\colon [0,1]\to \CBI(X)\) defined by
\begin{equation*}
t\mapsto \Biggl\{
\begin{aligned} 
\Phi_{\sigma\tau}(t)\colon X\times X\times [0,1] &\to X \\
(x,y,s)  &\mapsto \phi(\sigma_{xy}(s), \tau_{xy}(s), t)
\end{aligned}
\end{equation*}
is a geodesic in \((\CBI(X), D_o)\) connecting \(\sigma\) to \(\tau\). Moreover, \(\Phi\colon \CBI{(X)}\times\CBI{(X)}\times [0,1]\to \CBI{(X)}\) defined by
\((\sigma, \tau, t)\mapsto \Phi_{\sigma\tau}(t)\)
is a conical bicombing on \((\CBI(X), D_o)\). 
\end{Lem} 

\begin{proof}
Fix \(\sigma, \tau\in\CBI{(X)}\) and \(t\in [0,1]\). Letting \(\upsilon\coloneqq \Phi_{\sigma\tau}(t)\) and 
using that \(\phi\) satisfies \eqref{eq:conicalDEF}, we obtain 
\begin{equation}\label{eq:main-ineq}
d(\upsilon_{xy}(s), \upsilon_{x'y'}(s'))\leq (1-t) d(\sigma_{xy}(s), \sigma_{x'y'}(s'))+td(\tau_{xy}(s), \tau_{x'y'}(s'))
\end{equation}
for all \(x\), \(y\), \(x'\), \(y'\in X\) and all \(s\), \(s'\in [0,1]\). In particular, \(d(\upsilon_{xy}(s),\upsilon_{xy}(s'))\leq \abs{s-s'} d(x,y)\). Now, exactly the same argument as in the proof of Lemma~\ref{lem:trivial1} shows that \(\upsilon\) defines a bicombing on \(X\). Moreover, since the biombings \(\sigma\) and \(\tau\) are conical, it follows immediately from \eqref{eq:main-ineq} that \(\upsilon\) is conical was well. As a result, \(\Phi_{\sigma\tau}\colon [0,1]\to \CBI(X)\) is well-defined. Next, we show that it is a geodesic. Since \(\phi\) is a bicombing, we have
\[
d(\Phi_{\sigma\tau}(t)(x,y,s),\Phi_{\sigma\tau}(t')(x,y,s))=\abs{t-t'} d(\sigma_{xy}(s), \tau_{xy}(s))
\]
for all \(x\), \(y\in X\) and all \(s\in [0,1]\). Hence, \(D_o(\Phi_{\sigma\tau}(t), \Phi_{\sigma\tau}(t'))=\abs{t-t'}D_o(\sigma, \tau)\) for all \(t\), \(t'\in [0,1]\), as desired.
To finish the proof we need to show that \(\Phi\) is conical. Notice that
\[
d(\Phi_{\sigma\tau}(t)(x,y,s),\Phi_{\sigma'\tau'}(t)(x,y,s))\leq (1-t) d(\sigma_{xy}(s), \sigma'_{xy}(s))+td(\tau_{xy}(s), \tau'_{xy}(s))
\]
for all \(x\), \(y\in X\) and all \(s\in [0,1]\). Consequently, \(D_o(\Phi_{\sigma\tau}(t), \Phi_{\sigma'\tau'}(t))\leq (1-t) D_o(\sigma, \sigma')+t D_o(\tau, \tau')\), as was to be shown. 
\end{proof}

Clearly, \(\CBI(X)\) also admits conical bicombings for other choices of metrics. 
We do not know if there exists a conical bicombing on \((\CBI(X), D_o)\) which is not equal to \(\Phi\) for any bicombing \(\phi\) on \(X\). In other words, we do not know whether \(\phi\mapsto \Phi\) defines a surjective map from \(\CBI(X)\) to \(\CBI(\CBI(X))\). We conclude this subsection with the following straightforward result which states that \(\CBI(X)\) is compact whenever \(X\) is proper. 

\begin{Lem}\label{lem:compact}
Let \(X\) is a proper metric space. Then \(D_o\) induces the topology of uniform convergence on compact sets. In particular, \((\CBI(X), D_{o})\) is a compact metric space for all \(o\in X\). 
\end{Lem} 
\begin{proof}
Let \(K\subset X\) be a compact subset, \(\sigma\in \CBI(X)\) and \(\varepsilon >0\). We put 
\[
U(K, \sigma, \varepsilon)\coloneqq \Bigl\{\tau\in \CBI(X) : \sup_{x,y\in K} \norm{\sigma_{xy}-\tau_{xy}}_\infty <\varepsilon \Bigr\}.
\]
There exists \(k\geq 0\) such that \(K\subset B_{2^{k}}(o)\), and so \( U_{D_o}(\sigma, 3^{-k} \varepsilon) \subset U(K, \sigma, \varepsilon)\), where \(U_{D_o}(\sigma, 3^{-k} \varepsilon)\) denotes the open ball with respect to \(D_o\) with center \(\sigma\) and radius \(3^{-k} \varepsilon\). The sets \(U(K, \sigma, \varepsilon)\) form a basis of the topology \(\mathcal{T}_K\). Hence, by the above \(\mathcal{T}_{D_o}\subset \mathcal{T}_K\). Next, we show the other direction. Let \(\sigma\in \CBI(X)\) and \(\varepsilon >0\) be given. Choose \(k_0\geq 0\) such that \((\tfrac{2}{3})^{k_0} < \varepsilon\). We put \(K\coloneqq B_{2^{k_0}}(o)\). Notice that for all \(k\geq k_0\),
\[
\sup_{x,y\in B_{2^k}(o)} 3^{-k}\,\norm{\sigma_{xy}-\tau_{xy}}_\infty \leq \sup_{x,y\in B_{2^k}(o)} 3^{-k}\cdot 2^{k} < \varepsilon.
\]
Hence, \(U(K, \sigma, \varepsilon)\subset U_{D_o}(\sigma, \varepsilon)\), and as a result, \(\mathcal{T}_K\subset \mathcal{T}_{D_o}\), as desired. Since \(\mathcal{T}_{D_o}=\mathcal{T}_K\), it follows immediately from the Arzelà-Ascoli theorem that \((\CBI(X), D_o)\) is compact metric space for all \(o\in X\).
\end{proof}

\subsection {A fixed point result and its applications} The following proposition is due Kijima \cite{MR881533}. It can be proved by slightly adapting the proof of a well-known result from the fixed point theory of Banach spaces by Mitchell \cite{MR267414}. 

\begin{Prop}[Theorem 1 of \cite{MR881533}]\label{lem:fixed}
Let \(X\) be a compact metric space admitting a conical  bicombing. If \(S\) is a left reversible semigroup consisting of 1-Lipschitz self-maps of \(X\), then \(S\) has a common fixed point in \(X\), that is, there is \(x_\ast\in X\) such that \(f(x_\ast)=x_\ast\) for all \(f\in S\).  
\end{Prop}

Here, a semigroup \(S\) is \textit{left reversible} if for all \(a, b \in S\), there exist \(c, d\in S\) such that \(ac=bd\). For instance, every group and every abelian semigroup is left reversible. By the use of Lemmas \ref{lem:C-biBicombing} and \ref{lem:compact}, and Proposition \ref{lem:fixed}, certain results of \cite{MR3563261}, \cite{MR3096307} may be derived via straightforward fixed points arguments. For example:

\begin{Lem}\label{lem:reversible}
Let \(X\) be a proper metric space admitting a conical bicombing \(\phi\). Then \(X\) also admits a reversible conical bicombing. Moreover, the subset \(\RCBI(X)\subset \CBI(X)\)  of all reversible conical bicombings on \(X\) is closed and \(\Phi\)-convex. 
\end{Lem}

\begin{proof}
We define \(r\colon \CBI(X) \to \CBI(X)\) by
\[
\sigma\mapsto \Biggl\{
\begin{aligned}
r(\sigma)\colon X\times X\times [0,1]&\to X \\
(x,y,t)&\mapsto \phi(\sigma_{xy}(t), \sigma_{yx}(1-t), \tfrac{1}{2}).
\end{aligned}
\]
Fix \(o\in X\). It is easily seen that \(r\) is \(1\)-Lipschitz with respect to \(D_{o}\). Since  \((\CBI(X), D_o)\) is a compact metric space with a conical  bicombing, see Lemmas \ref{lem:C-biBicombing} and \ref{lem:compact}, it follows from Proposition \ref{lem:fixed} that there exists \(\sigma_\ast\in \CBI(X)\) such that \(r(\sigma_\ast)=\sigma_\ast\). By construction, \(\sigma_\ast\) is reversible. Next, we show that \(\RCBI(X)\) is \(\Phi\)-convex, where \(\Phi\) is defined as in Lemma~\ref{lem:C-biBicombing}.
If \(\sigma\), \(\tau\in \RCBI(X)\), then 
\[\phi(\sigma_{xy}(s), \tau_{xy}(s), t)=\phi(\sigma_{yx}(1-s), \tau_{yx}(1-s), t)\]
and thus \(\Phi_{\sigma\tau}(t)(x,y,s)=\Phi_{\sigma\tau}(t)(y,x,1-s)\) for all \((x,y,s)\in X\times X \times [0,1]\). Hence, \(\RCBI(X)\) is \(\Phi\)-convex.
To finish the proof we need to show that \(\RCBI(X)\) is closed. Let \((\sigma^{(n)})\) be a sequence of reversible conical bicombings converging to \(\sigma\in \CBI(X)\) as \(n\to \infty\). 
Fix \((x,y,t)\in X\times X\times [0,1]\). Since each \(\sigma^{(n)}\) is reversible, it follows that
\begin{align*}
d(\sigma_{xy}(t), \sigma_{yx}(1-t)) &\leq d(\sigma_{xy}(t), \sigma_{xy}^{(n)}(t))\\
&+d(\sigma_{yx}^{(n)}(1-t),\sigma_{yx}(1-t)).
\end{align*}
Choose \(k_0\geq 1\) such that \(x,y\in B_{2^{k_0}}(o)\). By the above,
\[d(\sigma_{xy}(t), \sigma_{yx}(1-t)) \leq 2 \cdot 3^{k_0} \cdot D_{o}(\sigma, \sigma^{(n)})\] for all \(n\geq 1\). This implies that \(\sigma\) is reversible, as desired. 
\end{proof}

A bicombing \(\sigma\) is called \textit{\(\Iso(X)\)-equivariant}
if  \(f(\sigma(x,y,t))=\sigma( f(x), f(y), t)\) for every isometry \(f\colon X\to X\), all \(x\), \(y\in X\) and all \(t\in [0,1]\). In \cite[Proposition 3.8]{MR3096307}, Lang proved that every injective metric space \(X\) admits an \(\Iso(X)\)-equivariant reversible conical bicombing. Using the moduli space \(\CBI(X)\) and Proposition~\ref{lem:fixed}, we can show that the analogous result also holds for every compact metric space which admits a conical bicombing. This seems to be of independent interest and does not follow from Lang's result.

\begin{Lem}\label{lem:IsoInv}
Let \(X\) be a compact metric space. If \(X\) admits a conical  bicombing, then \(X\) also admits an \(\Iso(X)\)-equivariant reversible conical  bicombing. 
\end{Lem}

\begin{proof}
Due to Proposition \ref{lem:fixed} there exists \(o\in X\) such that \(f(o)=o\) for every isometry \(f\) of \(X\).  For each isometry \(f\colon X\to X\) the map \(F\colon \CBI(X) \to \CBI(X)\) defined by
\begin{equation*}
\sigma\mapsto \Biggl\{
\begin{aligned}
F(\sigma)\colon X\times X\times [0,1]&\to X \\
(x,y,t)&\mapsto f^{-1}( \sigma( f(x), f(y), t))
\end{aligned}
\end{equation*}
is an isometric embedding with respect to \(D_o\) and \(\RCBI(X)\) is \(F\)-invariant. Since \((\CBI(X), D_o)\) and \((\RCBI(X), D_o)\) are compact metric spaces, see Lemma~\ref{lem:compact} and \ref{lem:reversible}, a classical result \cite{freudenthal1936dehnungen} tells us that the maps \(F\) and \(F|_{\RCBI(X)}\) are isometries. Because of Lemma \ref{lem:reversible}, the compact metric space \((\RCBI(X), D_o)\) admits a conical  bicombing. Hence, by virtue of Proposition \ref{lem:fixed} there exists \(\sigma_\ast\in \RCBI(X)\) such that \(F(\sigma_\ast)=\sigma_\ast\) for every map \(F\) defined as above. By construction, \(\sigma_\ast\) is an \(\Iso(X)\)-equivariant reversible conical  bicombing on \(X\).
\end{proof}

\section{Constructing new conical bicombings from old ones}\label{sec:constr}
\subsection{Preparatory lemmas}
Let \(X\) denote a metric space admitting a conical bicombing \(\sigma\). 
In what follows, we develop tools that allow us to construct new bicombings starting from \(\sigma\). Fix \(n\geq 1\) and \(\tau\in \CBI(X)\). For all \(x,y\in X\) we set \(\upsilon_{xy}(n; 0)\coloneqq x\), \(\upsilon_{xy}(n;n)\coloneqq y\), and
\begin{equation}\label{eq:defM}
\upsilon_{xy}(n; i)\coloneqq \sigma\bigl(\tau_{xy}\big(\tfrac{i-1}{n}\big),\tau_{xy}\big(\tfrac{i+1}{n}\big), \tfrac{1}{2}\bigr)
\end{equation}
for all \(i=1, \dots, n-1\). Let \(\upsilon_\sigma(n; \tau)\) denote the map \(X\times X\times [0,1]\to X\) defined by 
\begin{equation}\label{eq;defQ}
(x,y, (1-\lambda)\tfrac{i}{n}+\lambda\tfrac{i+1}{n})\mapsto \sigma(\upsilon_{xy}(n; i),\upsilon_{xy}(n; i+1), \lambda) 
\end{equation}
for all \(i=1, \dots, n-1\) and all \(\lambda\in [0,1]\).

\begin{Lem}\label{lem:basicLemma}
The map \(\upsilon_\sigma(n; \tau)\) is a conical bicombing, and if \(\sigma\) is consistent, then \(\upsilon_\sigma(n; \sigma)=\sigma\).
\end{Lem}
\begin{proof}
We abbreviate \(\upsilon\coloneqq\upsilon_\sigma(n; \tau)\) and \(t_i\coloneqq\frac{i}{n}\) for \(i=0, \ldots, n\). 
Using \eqref{eq:defM}, we obtain
\begin{equation}\label{eq:goe}
\begin{split}
&d(x, \upsilon_{xy}(n; i)) \leq \frac{t_{i-1}}2\, d(x,y)+\frac{t_{i+1}}{2} \,d(x,y), \\
&d(y, \upsilon_{xy}(n; i)) \leq \frac{1-t_{i-1}}{2}\, d(x,y)+ \frac{1-t_{i+1}}{2} \, d(x,y).
\end{split}
\end{equation}
As \(d(x,y) \leq d(x, \upsilon_{xy}(n; i))+d(\upsilon_{xy}(n; i),y)=d(x,y)\), the inequalities in \(\eqref{eq:goe}\) are equalities. 
Since
\[
d(\upsilon_{xy}(n; i), \upsilon_{xy}(n; i+1))\leq \frac{1}{n}\, d(x,y),
\]
we obtain \(d(\upsilon_{xy}(n; i), \upsilon_{xy}(n; j))=\abs{t_i-t_j}d(x,y)\) for all \(i, j=0, \dots, n\). Thus, \(\upsilon\) is a bicombing. 

We proceed to show that \(\upsilon\) satisfies \eqref{eq:conicalDEF}. Let \(t\in \bigl[t_i, t_{i+1}\bigr]\). Clearly, \(t=(1-\lambda) t_i+\lambda t_{i+1}\) for some \(\lambda\in [0,1]\).
Let \(x,y,z\in X\). We estimate
\begin{align}\label{eq:ONE}
d(\upsilon_{xy}(t), \upsilon_{xz}(t)) &\leq(1-\lambda) d(\upsilon_{xy}(n; i), \upsilon_{xz}(n; i)) \nonumber \\
&+\lambda \,d(\upsilon_{xy}(n; i+1), \upsilon_{xz}(n; i+1)).
\end{align}
By virtue of \eqref{eq:defM}, it follows that
\begin{align}\label{eq:TWO}
d(\upsilon_{xy}(n; j), \upsilon_{xz}(n; j)) &\leq \frac{1}{2} \,d(\tau_{xy}(t_{j-1}), \tau_{xz}(t_{j-1}))+ \frac{1}{2} \,d(\tau_{xy}(t_{j+1}), \tau_{xz}(t_{j+1})) \nonumber\\
&\leq t_j\,d(y,z) 
\end{align}
for all \(j=0, \dots, n\). By combining \eqref{eq:ONE} and \eqref{eq:TWO}, we find that \(\upsilon\) satisfies \eqref{eq:conicalDEF}, as desired. For the moreover part, it suffices to note that if \(\sigma\) is consistent, then \(\upsilon_{xy}(n; i)=\sigma_{xy}(t_i)\) for all \(i=0, \ldots, n\). 
\end{proof}

Now we are in a position to prove Lemma  \ref{lem:main1}, which is the main component of the proof of Theorem \ref{thm:consi}.
\begin{Lem}\label{lem:main1}
Let \(X\) be a proper metric space and suppose \(\sigma\) is a conical bicombing on \(X\). Let \(x,y\in X\).
Then for each integer \(n\geq 1\) there exist unique points \(\sigma_{xy}(n; i),  \textrm{ for } i=0, \ldots, n,\)
such that \(\sigma_{xy}(n; 0)=x\), \(\sigma_{xy}(n; n)=y\), and
\begin{equation}\label{eq:empty}
\sigma_{xy}(n; i)=\sigma\bigl( \sigma_{xy}(n; i-1), \sigma_{xy}(n; i+1), \tfrac{1}{2}\bigr)
\end{equation}
for all \(i=1, \dots, n-1\). Moreover, \(\sigma^{(n)}\colon X\times X\times [0,1]\to X\) defined by
\begin{equation*}
\sigma^{(n)}\bigl(x,y,(1-\lambda)\tfrac{i}{n}+\lambda\tfrac{i+1}{n}\bigr)\coloneqq\sigma(\sigma_{xy}(n; i), \sigma_{xy}(n; i+1), \lambda)
\end{equation*}
for all \(x,y\in X\), all \(\lambda\in [0,1]\), and all \(i=0, \dots, n-1\), is a conical  bicombing.  
\end{Lem}

\begin{proof}
To begin, we prove that the points \(\sigma_{xy}(n; i)\) are unique. Suppose that \(p_0,  \dots, p_n\in X\) satisfy \(p_0=x\), \(p_n=y\) and \(p_i=\sigma(p_{i-1}, p_{i+1}, \tfrac{1}{2})\) for all \(i=1, \dots, n-1\). We abbreviate \(d_i\coloneqq d(\sigma_{xy}(n; i), p_i)\) and \(d\coloneqq\max\{ d_i : i=0, \ldots, n \}\). Plainly, 
\[d_i\leq \frac{1}{2}d_{i-1}+\frac{1}{2}d_{i+1}\leq \frac{1}{4} d_{i-2}+\frac{1}{4} d_i+\frac{1}{2}d_{i+1} \leq \dots \leq\frac{1}{2^i} d_0+\big(1-\frac{1}{2^i}\big) d,\]
for all \(i=1, \dots, n-1\). Hence, \(d=0\), as desired. 

In the following, we show that the points \(\sigma_{xy}(n; i)\) with the desired properties exist. 
For \(\tau\in \CBI(X)\) let \(\upsilon_\sigma(n; \tau)\) be defined as in \eqref{eq;defQ}.
By Lemma \ref{lem:basicLemma}, it follows that \(\upsilon_\sigma(n; \tau)\) is a conical  bicombing. Fix \(o \in X\). By construction, \(\tau\mapsto \upsilon_\sigma(n; \tau)\) is 1-Lipschitz with respect to \(D_o\). Indeed, 
\begin{align*}
D_o(\upsilon_\sigma(n; \tau), \upsilon_\sigma(n; \tau^\prime)) \leq \sup\bigl\{ 3^{-k}\, d(\tau_{xy}(\tfrac{i}{n}), \tau^\prime_{xy}(\tfrac{i}{n})) : k\geq 0,\, x,y\in B_{2^k}(o), \, i\in [n] \bigr\}
\end{align*}
where \([n]\coloneqq\{0, \ldots, n\}\), and therefore
\[
D_o(\upsilon_\sigma(n; \tau), \upsilon_\sigma(n; \tau^\prime)) \leq D_o(\tau, \tau^\prime)
\]
for all \(\tau, \tau^\prime\in \CBI(X)\). Since \((\CBI(X), D_o)\) is compact (see Lemma \ref{lem:compact}), Proposition \ref{lem:fixed} now gives us some \(\sigma_\ast\in \CBI(X)\) for which \(\upsilon_\sigma(n; \sigma_\ast)=\sigma_\ast\). Hence, the points \(\sigma_{xy}(n; i)\coloneqq\sigma_\ast(x,y, \tfrac{i}{n})\) have the desired properties. 
\end{proof}

Rather than using the tools of Section \ref{sec:FIXED}, Lemma~\ref{lem:main1} can also be established by direct  computations. Indeed, straightforward (but tedious) estimates show that the sequence \((x_k)\subset X^{n+1}\) with \(x_0\in X^{n+1}\) arbitrary, 
\[x_k^{(0)}\coloneqq x, \,\, x_k^{(n)}\coloneqq y, \, \text{  and  } \, x_k^{(i)}\coloneqq\sigma( x_{k-1}^{(i-1)}, x_{k-1}^{(i+1)}, \tfrac{1}{2})\] 
is convergent. Its limit fulfils \eqref{eq:empty} by construction. 

\subsection{Proof of Theorem \ref{thm:consi}}
Now, we have everything at hand to prove Theorem \ref{thm:consi}.
\begin{proof}[Proof of Theorem \ref{thm:consi}] Fix a reversible conical  bicombing \(\sigma\) on \(X\).
The existence of such a bicombing is guaranteed by Lemma \ref{lem:reversible}. For each \(n\geq 1\) let \(\sigma^{(n)}\colon X\times X\times [0,1]\to X\) denote the conical  bicombing constructed in Lemma \ref{lem:main1}. Notice that \(\sigma^{(1)}=\sigma\). Since the points \(\sigma_{xy}(n; i), i=0, \ldots, n,\) are unique, we find that
\begin{equation}\label{eq:ruler}
\sigma^{(n)}_{xy}\bigl((1-\lambda)\tfrac{i}{n}+\lambda\tfrac{i+k}{n}\bigr)=\sigma^{(k)}\bigl(\sigma^{(n)}_{xy}\big(\tfrac{i}{n}\big),\sigma^{(n)}_{xy}\big(\tfrac{i+k}{n}\big), \lambda\bigr)
\end{equation}
for all \(0\leq i \leq i+k \leq n\) and all \(\lambda\in [0,1]\).
We define \(\gamma^{(n)}\colon X\times X\times [0,1]\to X\) by setting \(\gamma^{(n)}(x,x,t)\coloneqq x\) for all \(x\in X\) and all \(t\in [0,1]\), and
\begin{equation*}
\gamma^{(n)}(x,y,t)\coloneqq\sigma^{(i)}(x,y,t) \quad \quad \textrm{ if } d(x,y)\in \bigl(\frac{i-1}{n}, \frac{i}{n}\bigr]
\end{equation*}
Clearly, \(\gamma^{(n)}\) is a  bicombing. We remark that \(\gamma^{(n)}\) is not necessarily continuous with respect to the product topology on \(X\times X \times [0,1]\). 

Fix a free ultrafilter \(\mathfrak{U}\) on the positive integers (we refer to \cite[p. 78]{MR1744486} for the definition and basic properties of ultrafilters). Now, let \(\gamma\colon X\times X \times [0,1]\to X\) be defined by
\begin{equation*}
\gamma(x,y,t)\coloneqq\lim_{\mathfrak{U}} \gamma^{(n)}(x,y,t).
\end{equation*}
By construction, \(\gamma\) is a reversible bicombing. In the following, we show that \(\gamma\) has the desired properties. First, we prove that \(\gamma\) is consistent. 
To this end, let \(x,y\in X\) with \(x\neq y\), and \(p\coloneqq\gamma_{xy}(s)\) and \(q\coloneqq\gamma_{xy}(t)\) for \(0 \leq s < t \leq 1\) be given. By construction, 
\[
d\bigl(\gamma_{xy}((1-\lambda)s+\lambda t), \gamma_{pq}(\lambda)\bigr)=\lim_{\mathfrak{U}} d\bigl(\gamma^{(n)}_{xy}((1-\lambda)s+\lambda t), \gamma^{(n)}_{pq}(\lambda)\bigr).
\]
Letting \(p_n\coloneqq\gamma^{(n)}_{xy}(s)\) and \(q_n\coloneqq\gamma^{(n)}_{xy}(t)\), we obtain that
\begin{align*}
d(\gamma_{xy}((1-\lambda)s+\lambda t), \gamma_{pq}(\lambda)) &\leq \lim_{\mathfrak{U}}d\bigl(\gamma^{(n)}_{xy}((1-\lambda)s+\lambda t), \gamma^{(n)}_{p_nq_n}(\lambda)\bigr) \\
&+\lim_{\mathfrak{U}} d\bigl(\gamma^{(n)}_{p_nq_n}(\lambda), \gamma^{(n)}_{pq}(\lambda)\bigr).
\end{align*}
Notice that \(d(p_n, q_n)=d(p_m, q_m)\) for all \(n,m\geq 1\), and so \(d(p,q)=d(p_n, q_n)\) for all \(n\geq 1\). Hence, using the definition of \(\gamma^{(n)}\), we obtain
\begin{equation*}
\begin{split}
\lim_{\mathfrak{U}} d\bigl(\gamma^{(n)}_{p_nq_n}(\lambda), \gamma^{(n)}_{pq}(\lambda)\bigr) &\leq (1-\lambda)\lim_{\mathfrak{U}}d(p_n,p)+\lambda\lim_{\mathfrak{U}}d(q_n,q) \leq 0.
\end{split}
\end{equation*}
Thus, to prove that \(\gamma\) is consistent, it suffices to show that
\begin{equation}\label{eq:goal1}
\lim_{\mathfrak{U}}d\bigl(\gamma^{(n)}_{xy}((1-\lambda)s+\lambda t), \gamma^{(n)}_{p_nq_n}(\lambda)\bigr)=0.
\end{equation}
Fix \(n\geq 1\) and denote by \(m,k\geq 1\) the unique integers such that 
\[d(x,y)\in \bigl(\frac{m-1}{n}, \frac{m}{n}\bigr], \quad \quad d(p_n,q_n)\in \bigl(\frac{k-1}{n}, \frac{k}{n}\bigr].\]
We set \(A\coloneqq\big\{\sigma^{(m)}_{xy}(\frac{i}{m}) : i=0, \dots, m \big\}\). Choose \(p'_n\coloneqq\sigma^{(m)}_{xy}(\frac{i}{m})\in A\) such that \(d(p_n, p'_n)+d(q_n,q'_n)\leq d(x,y)/m \leq 1/n\),
where \(q'_n\coloneqq\sigma^{(m)}_{xy}(\frac{i+k}{m})\). We estimate
\[
d\bigl(\gamma_{p_nq_n}^{(n)}(\lambda), \sigma_{p'_nq'_n}^{(k)}(\lambda)\bigr)=d\bigl(\sigma_{p_nq_n}^{(k)}(\lambda), \sigma_{p'_nq'_n}^{(k)}(\lambda)\bigr)\leq \frac{1}{n}
\]
and therefore
\[
d\bigl(\gamma^{(n)}_{xy}((1-\lambda)s+\lambda t), \gamma^{(n)}_{p_nq_n}(\lambda)\bigr)\leq d\bigl(\sigma^{(m)}_{xy}((1-\lambda)s+\lambda t), \sigma_{p'_nq'_n}^{(k)}(\lambda)\bigr)+\frac{1}{n}.
\]
By virtue of \eqref{eq:ruler}, we arrive at
\begin{equation*}\label{eq:consistency1}
d\bigl(\gamma^{(n)}_{xy}((1-\lambda)s+\lambda t), \gamma^{(n)}_{p_nq_n}(\lambda)\bigr)\leq \frac{2}{n};
\end{equation*}
hence, \eqref{eq:goal1} follows. Thus, \(\gamma\) is a consistent bicombing, as claimed. Since every geodesic of \(\gamma^{(n)}\) is a \(\sigma^{(i)}\)-geodesic for some \(i\geq 1\), each map \(\gamma_{xy}(\cdot)\) is a straight geodesic. Similarly, since \(\gamma\) is consistent and by the definition of \(\gamma^{(n)}\), we find that \(t\mapsto d(\gamma_{xy}(t), \gamma_{x^\prime y^\prime}(t))\) is convex on \([0,1]\) whenever \(x\), \(y\), \(x^{\prime}\), \(y^{\prime}\in X\) satisfy \(d(x,y)=d(x^\prime, y^\prime)\). Notice that if \(\sigma\) is \(\Iso(X)\)-equivariant, then each \(\sigma^{(n)}\) is \(\Iso(X)\)-equivariant, and it follows that \(\gamma\) is \(\Iso(X)\)-equivariant as well. Due to Lemma~\ref{lem:IsoInv} and \cite[Proposition 3.8]{MR3096307}, we may suppose that \(\sigma\) is \(\Iso(X)\)-equivariant whenever \(X\) is compact or injective. This completes the proof. 
\end{proof}

\begin{Expl}
Let \((\sigma^{(n)})\) be the sequence of conical bicombings as in the proof of Theorem \ref{thm:consi}.  If \((\sigma^{(n)})\) converges,
then it is immediate that the limit is a consistent conical bicombing on \(X\). 
Unfortunately, standard estimates only show that \[D_o(\sigma^{(n)}, \sigma^{(n+1)}) \leq  \frac{1}{n+1}.\] 
We do not know if the sequence \((\sigma^{(n)})\) is convergent. 
\end{Expl}

\section{Boundary constructions}

\subsection{Boundaries at infinity}\label{sec:construction}

Throughout this subsection, let \(X\) be a complete metric space and \(\gamma\) a consistent bicombing on \(X\) 
such that \(t\mapsto d(\gamma_{xy}(t), \gamma_{x' y'}(t))\) is convex on \([0,1]\) for all \(x\), \(y\), \(x'\), \(y'\in X\) with \(d(x,y)=d(x', y')\).  If \(X\) is proper and admits a conical bicombing then the existence of such a bicombing \(\gamma\) is guaranteed by Theorem~~\ref{thm:consi}. Moreover, every Busemann space clearly admits such a bicombing. The aim of this section is the adapt the boundary construction for Busemann spaces which is given in \cite{MR1425125} to this more general setting. For consistent conical bicombings this has already been carried out in \cite{MR3370039}. We caution the reader that it is not known whether \(\gamma\) satisfies the conical inequality \eqref{eq:conicalDEF}. Hence, in the following arguments we do not have \eqref{eq:conicalDEF} at our disposal. 

Given \(o\), \(x\in X\) we define \(\rho_{ox}\colon \R_+\to X\) as follows. If \(t\geq d(o,x)\) then we set \(\rho_{ox}(t)=x\)  and \(\rho_{ox}(t)=\gamma_{ox}\bigl(t/d(o,x)\bigr)\) otherwise. We will often use the following elementary estimate as a substitute for \eqref{eq:conicalDEF}.

\begin{Lem}\label{lem:technical-1}
Let \(o\), \(x\), \(y\in X\) be distinct. Then
\begin{equation}\label{eq:estimate-222}
d(\rho_{ox}(t), \rho_{oy}(t)) \leq 2 t \cdot \frac{d(x,y)}{\min\bigl\{d(o,x), d(o,y)\bigr\}}
\end{equation}
for all \(t\in \R_+\) satisfying \(t \leq \min\bigl\{d(o,x), d(o,y)\bigr\}\).
\end{Lem}

\begin{proof}
We may suppose that \(d(o,x)\leq d(o,y)\). We put \(y'\coloneqq \rho_{oy}\bigl(d(o,x)\bigr)\). Since \(\rho_{oy}(t)=\rho_{oy'}(t)\), it follows that \(d(\rho_{ox}(t), \rho_{oy}(t)) \leq \frac{t}{d(o,x)} d(x,y')\). 
By construction, \(d(y,y')=\abs{d(y,o)-d(x,o)}\leq d(x,y)\), and so the desired result follows by the triangle inequality. 
\end{proof}

Recall that a ray \(\xi\colon \R_+\to X\) is called \(\gamma\)-ray if \(\xi((1-\lambda) s +\lambda t)=\gamma(\xi(s), \xi(t), \lambda)\) for all \(0 \leq s \leq t\) and all \(\lambda\in [0,1]\). For each \(o\in X\) we let \((\partial_\gamma X)_o\) denote the set of all \(\gamma\)-rays issuing from \(o\). The following lemma shows that \((\partial_\gamma X)_o\) and \((\partial_\gamma X)_p\) are bijectively equivalent. Having Lemma~\ref{lem:technical-1} at hand, it can be proven by slightly adapting the arguments from \cite{MR1425125}.

\begin{Lem}\label{ref:unique-rep}
Let \(p\), \(o\in X\). Then for every \(\xi\in (\partial_\gamma X)_p\) there exists a unique \(\xi'\in (\partial_\gamma X)_o\) such that \(\xi\) and \(\xi'\) are asymptotic.
\end{Lem}

\begin{proof}
The uniqueness part follows directly from the fact that \(t\mapsto d(\xi(t), \xi^\prime(t))\) is convex for all \(\gamma\)-rays \(\xi\) and \(\xi^\prime\). 
Let \(\xi\in (\partial_\gamma X)_p\) and define \(x_n\coloneqq \xi(n)\) for all \(n\geq 1\). Fix \(t\in \R_+\). We claim that the sequence \(\big(\rho_{o x_n}(t)\big)\) converges. 
Suppose that \(N\geq t+3d(o,p)\) and let \(n\geq m \geq N\). Notice that \(x_m=\rho_{x_n p}(d(x_n, x_m))\) and 
\[
d(x_n, o)\geq d(x_n, p)-d(o,p)\geq d(x_n, x_m)+N-d(o, p)\geq d(x_n, x_m).
\]
 Hence, \(z\coloneqq \rho_{x_n o}(d(x_n, x_m))\) is well-defined. Using Lemma~\ref{lem:technical-1}, we obtain \(d(x_m, z) \leq 2 d(o, p)\), and so 
\[
d(z,o)\geq d(o, x_m)-d(x_m, z)\geq N-3d(o,p)\geq t.
\]
This implies that \(\rho_{o x_n}(t)=\rho_{oz}(t)\) and thus by applying Lemma~\ref{lem:technical-1},
\begin{equation}\label{eq:technical-2}
d(\rho_{o x_m}(t), \rho_{o x_n}(t))\leq 2t\cdot \frac{d(x_m, z)}{\min\{d(o, x_m), d(o, z)\}} \leq 4 t \cdot\frac{d(o, p)}{N- 3 d(o, p)}.
\end{equation}
Therefore, \(\big(\rho_{o x_n}(t)\big)\) is a Cauchy sequence, as desired. Letting
\[
\xi^\prime(t)=\lim_{n\to \infty} \rho_{o x_n}(t),
\] 
it follows that \(\xi^\prime\) is a \(\gamma\)-ray emanating from \(o\). It remains to show that \(\xi\) and \(\xi^\prime\) are asymptotic. Fix \(t\in \R_+\). For \(n\geq 1\) sufficiently large, \(\xi(t)=\rho_{p x_n}(t)\) and \(t'\geq 0\), where \(t'\coloneqq \min\{d(p, x_n), d(o, x_n)\}-t\). Clearly, 
\[
d(\xi(t), \rho_{o x_n}(t))\leq d(\rho_{p x_n}(t), \rho_{x_n p}(t'))+d(\rho_{x_n p}(t'), \rho_{x_n o}(t'))+d(\rho_{x_n o}(t'), \rho_{o x_n}(t))
\]
and thus, \(d(\xi(t), \rho_{o x_n}(t)) \leq d(o, p)+d(\rho_{x_n p}(t'), \rho_{x_n o}(t'))\). Hence, by invoking Lemma~\ref{lem:technical-1}, it follows that \(d(\xi(t), \rho_{o x_n}(t)) \leq 3 d(o, p)\) for \(n\geq 1\) sufficiently large. This gives \(d(\xi(t), \xi^\prime(t)) \leq 3 d(o,p)\), as was to be shown. 
\end{proof}

Let \(\partial_\gamma X\) denote the set of equivalence classes of mutually asymptotic \(\gamma\)-rays and set \(\overline{X}_\gamma \coloneqq X \cup \partial_\gamma X\). Now, we are ready to define a family of metrics \((\bar{d}_o)_{o\in X}\) on \(\overline{X}_\gamma\). By Lemma~\ref{ref:unique-rep}, it follows that for every pair \((o, \bar{x})\in X\times\partial_\gamma X\) there exists a unique \(\rho_{o \bar{x}}\in (\partial_\gamma X)_o\) such that \([\rho_{o \bar{x}}]=\bar{x}\). By the triangle inequality, \(d(\rho_{o \bar{x}}(t), \rho_{o \bar{y}}(t))\leq 2 t\) for all \(\bar{x}\), \(\bar{y}\in \overline{X}_\gamma\). Hence, \(\bar{d}_o\colon \overline{X}_\gamma \times \overline{X}_\gamma \to \R\) given by 
\[
\bar{d}_o(\bar{x}, \bar{y})=\int_{0}^{\infty} d(\rho_{o \bar{x}}(t), \rho_{o \bar{y}}(t)) e^{-t}\, dt
\]
satisfies \(\bar{d}_o(\bar{x}, \bar{y}) \leq 2\) and defines a metric on \(\overline{X}_\gamma\).

\begin{Lem}\label{lem:topology}
Let \(o\), \(p\in X\). Then \(\bar{d}_o\) and \(\bar{d}_p\) induce the same topology on \(\overline{X}_\gamma\).
\end{Lem}

\begin{proof}
For every \(x\in X\) there exists some \(\varepsilon_x>0\) such that the open ball \(\{ \bar{x}\in \overline{X}_\gamma : \bar{d}_p(x, \bar{x}) < \varepsilon_x\}\) is contained in \(X\). In particular, \(X\) is an open subset of \(\overline{X}_\gamma\) with respect to any metric \(\bar{d}_p\). One can easily show that the induced topology of \(X\subset \overline{X}_\gamma\) is equal to the metric topology of \(X\).

%Fix \(o\), \(p\in X\). In what follows, we show that \(\bar{d}_o\) and \(\bar{d}_p\) induce the same topology on \(\overline{X}_\gamma\).
%By the above, it suffices to show that they induce the same topology on \(\partial_\gamma X\). 
Suppose now that \(\bar{x}\in \partial_\gamma X\) and  \(\bar{x}_n\in\overline{X}_\gamma\), \(n\geq 1\), are such that \(\bar{d}_p(\bar{x}_n, \bar{x})\to 0\) as \(n\to \infty\). To conclude the proof, we need to show that \(\bar{d}_o(\bar{x}_n, \bar{x})\to 0\) as \(n\to \infty\). Fix \(\varepsilon >0\). Clearly, there exists \(N_0\in \N\) sufficiently large such that for all \(N\geq N_0\), 
\[
N-3 d(o, p)>0, \quad \frac{4 d(o, p)}{N-3 d(o, p)} \leq \varepsilon, \quad \text{ and } \quad \int_N^\infty 2t e^{-t} \, dt \leq \varepsilon.
\] 
We define \(N\coloneqq N_0+3 d(o, p)\) and put \(z\coloneqq \rho_{p \bar{x}}(N)\) and \(z_n\coloneqq \rho_{p \bar{x}_n}(N)\) for all \(n\geq N\). Since \(\bar{x}\in \partial_\gamma X\) and \(\bar{d}_p(\bar{x}_n, \bar{x})\to 0\) as \(n\to \infty\), there exists \(K \geq N\) sufficiently large such that for all \(n \geq K\), one has \(d(p, z_n)=N\) and \(d(z_n, z) \leq \varepsilon \). Now, \eqref{eq:technical-2} tells us that for all \(t\in [0, N_0]\),
\[
d(\rho_{oz_n}(t), \rho_{o \bar{x}_n}(t))\leq 4 t \cdot\frac{d(o, p)}{N- 3 d(o, p)} \leq t \varepsilon 
\]
for all \(n\geq K\), and the analogous estimate also holds for \(d(\rho_{o z}(t), \rho_{o \bar{x}}(t))\). As a result,
\begin{align*}
\bar{d}_o(\bar{x}_n, \bar{x})&\leq \int_0^{N_0} d(\rho_{o \bar{x}_n}(t), \rho_{o \bar{x}}(t)) e^{-t}\, dt+\int_{N_0}^\infty 2t e^{-t}\, dt \\
&\leq 2\varepsilon \int_0^{N_0} t e^{-t}\, dt+\int_0^{N_0} d(\rho_{o z_n}(t), \rho_{o z}(t)) e^{-t}\, dt +\varepsilon,
\end{align*}
which implies \(\bar{d}_o(\bar{x}_n, \bar{x}) \leq 2d(z_n, z)+3\varepsilon\). In particular, \(\bar{d}_o(\bar{x}_n, \bar{x}) \leq 5 \varepsilon\) for all \(n\geq K\). Since \(\varepsilon >0\) was arbitrary, we find that \(\bar{d}_o(\bar{x}_n, \bar{x})\to 0\) as \(n\to \infty\). This completes the proof. 
\end{proof}

\subsection{\(\mathcal{Z}\)-compactifications and proof of Theorem~\ref{thm:descombes-lang-gen}}

Let \(X\) be a proper metric space and \(\overline{X}\) a compactification of \(X\). We follow \cite{MR4007577} and say that \(\overline{X}\) is a \textit{\(\mathcal{Z}\)-compactification} of \(X\) if \(\overline{X}\setminus X\) is a \(\mathcal{Z}\)-set in \(\overline{X}\). We will need the following general fact about \(\mathcal{Z}\)-compactifications. 

\begin{Lem}\label{lem:general-topology}
Let \(X\) be a proper metric space. If \(X\) is an absolute retract, then any \(\mathcal{Z}\)-compactification of \(X\) is an absolute retract. 
\end{Lem}
Recall that a metrizable topological \(X\) space is an \textit{absolute retract} if whenever \(X\subset X^\prime\) is a closed subspace of a metrizable topological space \(X'\), then \(X\) is a retract of \(X'\), that is, there exists a continuous map \(r\colon X'\to X\) such that \(r(x)=x\) for all \(x\in X\). 
\begin{proof}[Proof of Lemma~\ref{lem:general-topology}]
This follows from a classical theorem of Hanner \cite[Theorem 7.2.]{MR43459}. We refer to the discussion surrounding Lemma~3.3 in \cite{MR4007577} for more information. 
We remark that in \cite{MR4007577} and \cite{MR43459} the authors work within the category of metrizable separable spaces. Thus, strictly speaking, it only follows that \(\mathcal{Z}\)-compactifications of \(X\) are absolute retracts in this category. However, it is well-known that any absolute retract in this category is also an absolute retract in the category of metrizable topological spaces. Indeed, this is a direct consequence of Tietze's extension theorem and the fact that every metrizable seperable space can be realized as a closed subset of \(\R^\infty\).   
\end{proof}

We conclude this section with the proofs of Theorem~\ref{thm:descombes-lang-gen} and Corollary~\ref{cor:ez-structure}.

\begin{proof}[Proof of Theorem~\ref{thm:descombes-lang-gen}]
Fix \(o\in X\) and let \(h\colon \overline{X}_\gamma \times [0,1]\to\overline{X}_\gamma\) be defined by \(h(\bar{x}, 0)=\bar{x}\) and \(h(\bar{x}, t)=\rho_{o \bar{x}}\bigl(\frac{1-t}{t}\bigr)\) for \(t\in (0,1]\). Clearly, \(h_t(\overline{X}_\gamma)\subset X\) whenever \(t\in (0,1]\) and using the metric \(\bar{d}_o\) it is easy to check that \(h\) is continuous. Hence, it follows that \(\partial_\gamma X\) is a \(\mathcal{Z}\)-set in \(\overline{X}_\gamma\). By invoking the Arzelà–Ascoli theorem, we find that \(\overline{X}_\gamma\) is compact and thus \(\overline{X}_\gamma\) is a \(\mathcal{Z}\)-compactification of \(X\). Now, since \(X\) admits a conical bicombing, it is strictly equiconnected, and thus by a classical result due to Himmelberg (see \cite[Theorem 4]{MR195038}), it follows that \(X\) is an absolute retract. Therefore, by Lemma~\ref{lem:general-topology} we obtain that \(\overline{X}_\gamma\) is an absolute retract as well.
\end{proof}

\begin{proof}[Proof of Corollary~\ref{cor:ez-structure}]
Let \(G\) denote a group which acts geometrically an a proper metric space \(X\) admitting a conical bicombing. Let \(\gamma\) be a consistent bicombing on \(X\) satisfying the properties stated in Theorem~\ref{thm:consi}. Fix \(o\in X\). Let \(\overline{X}_\gamma\) be constructed as in Section~\ref{sec:construction} and equip it with the topology induced by \(\bar{d}_o\). We claim that \((\overline{X}_\gamma, \partial_\gamma X)\) defines a \(\mathcal{Z}\)-structure of \(G\). Clearly, \(\overline{X}_\gamma\) is compact and \eqref{it:2} holds. By invoking Theorem~\ref{thm:descombes-lang-gen}, we obtain \eqref{it:1}. In the following, we show \eqref{it:3}. Let \(C\subset X\) be a compact subset and \(R_0>0\) a real number such that \(C\subset B_R\) for all \(R\geq R_0\), where \(B_R\subset X\) denotes the closed ball of radius \(R\) centered at \(o\). By the use of Lemma~\ref{lem:technical-1}, it follows that 
\[
\bar{d}_o( gx, gy)\leq \frac{2\diam C}{R}\int_{0}^R t e^{-t}\, dt+2\int_{R}^\infty t e^{-t}\, dt \quad \quad \quad (x,y\in C)
\]
for all \(R\geq R_0\) and all \(g\in G\) which satisfy \(g B_R\cap B_R=\varnothing\). Let \(\mathcal{U}\) be an open cover of \((\overline{X}_\gamma, \bar{d}_o)\) and fix a Lebesgue number \(\delta\in \bigl(0,\tfrac{1}{10}\bigr)\) of \(\mathcal{U}\).
Fix \(R>\max\bigl\{ R_0,\, \tfrac{4}{\delta}\cdot\diam C,\, \log\bigl(\tfrac{4}{\delta^{2}}\bigr)\bigr\}\). By the above, if \(g\in G\) satisfies \(g B_R\cap B_R=\varnothing\), then the diameter of \(gC\) with respect to \(\bar{d}_o\) is smaller than \(\delta\) and thus \(gC\) is contained in some member of \(\mathcal{U}\). Hence, \eqref{it:3} follows, as \(g B_R\cap B_R=\varnothing\) for all but finitely many \(g\in G\). This proves that \((\overline{X}_\gamma, \partial_\gamma X)\) is a \(\mathcal{Z}\)-structure of \(G\). 
Finally, if \(X\) is injective then \(\gamma\) is equivariant with respect to the isometry group of \(X\), and thus the action of \(G\) on \(X\) can be extended to \(\overline{X}_\gamma\). Notice that \(\bar{d}_o(g\bar{x}, g\bar{y})=\bar{d}_{go}(\bar{x}, \bar{y})\) for all \(\bar{x}\), \(\bar{y}\in \overline{X}_\gamma\).  
Hence, because of Lemma~\ref{lem:topology} it follows that \(G\) acts by homeomorphisms on \(\overline{X}_\gamma\), and so \((\overline{X}_\gamma, \partial_\gamma X)\) is an \(E\mathcal{Z}\)-structure of \(G\), as desired. 
\end{proof}

\printbibliography

\end{document}